\documentclass[a4paper,reqno,10pt]{amsart}
\usepackage[utf8x]{inputenc}
\usepackage[colorlinks=true, pdfstartview=FitV, linkcolor=purple, citecolor=purple]{hyperref}
\usepackage{amsmath}
\usepackage{amsthm}
\usepackage[foot]{amsaddr}
\usepackage{amssymb}
\usepackage{academicons}
\usepackage{lineno}
\usepackage{cleveref}
\usepackage{bbm}
\usepackage{pictex}
\usepackage{tikz}
\usepackage{pgfplots}
\usepackage{subfigure}
\usetikzlibrary{shapes, positioning, arrows.meta,calc}
\usepackage[margin=0.8in]{geometry}
\usepackage{tabularx}
\usepackage{float}
\usepackage{import}
\usepackage{enumitem}

\usepackage{listings}
\usepackage{xcolor} 

\newtheorem{theorem}{Theorem}[section]
\newtheorem{lemma}[theorem]{Lemma}
\newtheorem{proposition}[theorem]{Proposition}

\newtheorem{remark}{Remark}

\pgfplotsset{compat = newest}

\newcommand\floor[1]{\left\lfloor #1 \right\rfloor}

\definecolor{orcidlogocol}{HTML}{A6CE39}
 
\title{On a tree of rational functions related to continued fractions}

\author{Niels Langeveld}
\address[Niels Langeveld]{Lehrstuhl f\"{u}r Mathematik und Statistik, Montanuniversit\"{a}t Leoben, Leoben, Austria}
\email{niels.langeveld@unileoben.ac.at}
\author{David Ralston}
\address[David Ralston]{
Mathematics, Computer \& Information Science, SUNY Old Westbury, New York, USA }
\email{ralstond@oldwestbury.edu}

\begin{document}
\begin{abstract}
    In this article, we present a binary tree with vertices given by rational functions $p(x)/q(x)$; the root and functional derivation of children are inspired by continued fractions. We prove some special properties of the tree. For example, the  zero solutions of the denominators $q(x)$ are all real negative numbers and are dense in $(-\infty,-1]$. For $x>0$ functions are non intersecting and form a dense subset of $(0,1)$. Furthermore, when evaluating the tree for positive rational values, the tree contains every rational in $(0,1)$ exactly once if and only if $x\in \mathbb{N}$. For $x=1$, one finds back the classical Farey tree which is related to regular continued fractions. In the last part, we will  make a similar tree in a similar way but for backward continued fractions. We highlight some similarities and differences.
\end{abstract}
\maketitle
\section{Introduction}
Let us first define the tree of interest which we will call the \textit{Farey polynomial tree} (an explanation of the derivation of this tree will be given in the beginning of \Cref{sec:cfs}). On the nodes we have rational functions $v(x)=p(x)/q(x)$. The root of the tree is $\frac{x}{x+1}$ and the two offspring of each node are found using the following two functions:
\begin{equation}\label{eqn: Phi}\begin{split}
\Phi_0\left( \frac{p(x)}{q(x)}\right) &= \begin{cases}
\frac{xq(x)}{xq(x)+p(x)} & p(0)\neq 0\\
\frac{q(x)}{q(x)+p(x)/x}& p(0)=0\end{cases}\\
\\
\Phi_1\left( \frac{p(x)}{q(x)}\right) &= \begin{cases}
\frac{xp(x)}{xq(x)+p(x)} & p(0)\neq 0\\
\frac{p(x)}{q(x)+p(x)/x}& p(0)=0.\end{cases}\end{split}\end{equation}

See Figure \ref{fig:FareyPoly} for the first few levels. These two functions are the inverse branches of a generalized Farey map and the tree generalizes the classical Farey tree which is found for $x=1$. Note that we define it by using inverse images of the Farey map and not by using mediants of neighboring fractions, but the trees are intimately related (see \cite{BI09}). The Farey tree is well studied and appears in many different branches of mathematics (see for example~\cite{AK22,BBDG24,BDM21,DKS25,DS07,KO86,LRS17} and the references therein). Curiously, Farey himself (a geologist) did not publish anything significant on the matter. It was Cauchy who proved one of the basic ideas of the Farey sequence and attributed it to Farey (see \cite[Notes on Chapter~3]{HW79}).

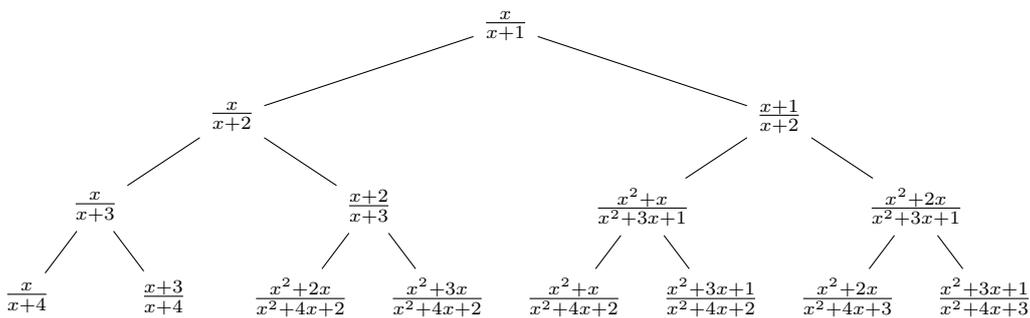
\begin{figure}[ht]
    \begin{center}

\begin{tikzpicture}[
    level 1/.style = {sibling distance=6cm},
    level 2/.style = {sibling distance=3cm},
    level 3/.style = {sibling distance=1.5cm},
    level distance          = 1cm,
    edge from parent/.style = {draw},
    scale=1.2
    ]

    \node {$\frac{x}{x+1}$}
    child{
      node {$\frac{x}{x+2}$}
      child{
        node {$\frac{x}{x+3}$}
        child{
          node {$\frac{x}{x+4}$}
        }
        child{
          node {$\frac{x+3}{x+4}$}
        }
      }
      child{
        node {$\frac{x+2}{x+3}$}
        child{
          node {$\frac{x^2+2x}{x^2+4x+2}$}
        }
        child{
          node {$\frac{x^2+3x}{x^2+4x+2}$}
        }
      }
    }
    child{
      node {$\frac{x+1}{x+2}$}
      child{
        node {$\frac{x^2+x}{x^2+3x+1}$}
        child{
          node {$\frac{x^2+x}{x^2+4x+2}$}
        }
        child{
          node {$\frac{x^2+3x+1}{x^2+4x+2}$}
        }
      }
      child{
        node {$\frac{x^2+2x}{x^2+3x+1}$}
        child{
          node {$\frac{x^2+2x}{x^2+4x+3}$}
        }
        child{
          node {$\frac{x^2+3x+1}{x^2+4x+3}$}
        }
      }
    };
\end{tikzpicture}
     
        \caption{The first four levels of the Farey polynomial tree.}\label{fig:FareyPoly}
    \end{center}
\end{figure}

We label each vertex in the tree with notation inspired by continued fractions. For $a_i \in \mathbb{Z}^+$:
\begin{equation}\label{eqn: vertex labels}
[a_1,a_2,\ldots,a_k]_x = \Phi_1^{a_1-1} \circ \Phi_0 \circ \Phi_1^{a_2-1} \circ \Phi_0 \circ \cdots \circ \Phi_1^{a_{k-1}-1}\circ \Phi_0 \circ \Phi_1^{a_k-1} \left( \frac{1}{1}\right)
\end{equation}

In this way the root of our tree is the vertex $[2]_x=[1,1]_x$, and in general $[a_1,\ldots,a_{k-1},a_k+1]_x=[a_1,\ldots,a_{k-1},a_k,1]_x$. Observe also that
\begin{align*}
    \Phi_0\left( [a_1,\ldots,a_k]_x\right) &= [1,a_1,\ldots,a_k]_x\\
    \Phi_1\left( [a_1,\ldots,a_k]_x\right) &= [a_1+1,a_2,\ldots,a_k]_x\\
\end{align*}The tree we present here has some very nice properties. We highlight some of them with the following two theorems; by ``subtree at vertex $[a_1,\ldots,a_k]_x$" we mean the tree generated by $\Phi_1$ and $\Phi_0$ rooted at vertex $[a_1,\ldots,a_k]_x$.

We have the following two main theorems.

\begin{theorem}\label{th:cfs}
For a fixed value of $x$, we consider the function which maps the tree to the set of values given by evaluating every vertex at this value of $x$.
    \begin{enumerate}
        \item Setting a value $x>0$ injectively maps the tree to a dense subset of $(0,1)$.
        \item Setting a value $x>0$ bijectively maps the tree to $\mathbb{Q} \cap (0,1)$ if and only if $x \in \mathbb{Z}^+$.
    \end{enumerate}
\end{theorem}
Note that from (1) in Theorem \ref{th:cfs} it follows that the rational functions in the tree don't intersect on $(0,\infty)$.

\begin{theorem}\label{th:dense}
    For the subtree $\Omega$ rooted at any vertex $[a_1,\ldots,a_k]_x$, if we set
    \[R=\left\{ x \, | \, \textrm{there is a } p(x)/q(x) \in \Omega \textrm{ for which }q(x)=0\right\}\]
    then $R$ is dense in $(-\infty,-1]$.
\end{theorem}

We begin \Cref{sec:cfs} with a derivation of this tree. For Theorem \ref{th:cfs} we use the close relationship between the tree and a family of continued fractions explained in Section \ref{sec:cfresults}. Theorem \ref{th:dense} is proven in \Cref{sec:dense} by inductively selecting a path in the tree whose vertices have poles converging to an arbitrary $\alpha<-1$; comments are included regarding the possibility of expanding this result to $\alpha \in (-1,0)$. A similar tree derived from the theory of backwards continued fractions is presented in \Cref{sec:other trees}.
We will see a similar behaviour in view of Theorem \ref{th:cfs} and a very different one in view of Theorem \ref{th:dense} for this tree.

\section{Forward Farey Tree}\label{sec:cfs}
\subsection{Background and Definition}
The \textit{Gauss map} $T:[0,1) \mapsto [0,1)$, given by
\[T(t) = \begin{cases}
    \frac{1}{t} - \floor{\frac{1}{t}} & t\neq 0\\
    0 & t=0,
\end{cases}\]
may be used to generate the \textit{regular continued fraction} expansion of $t$:
\[ t = \cfrac{1}{a_1+\cfrac{1}{a_2+\cfrac{1}{a_3+\ddots}}},\]
where $a_i =a(T^{i-1}(t))= \floor{1/T^{i-1}(t)}$. This map is the `sped-up' version of the \textit{Farey map} $F$, given by
\[ F(t) = \begin{cases}
    F_0(t) = \frac{1-t}{t} & x \geq 1/2\\
    F_1(t) = \frac{t}{1-t} & x \leq 1/2
\end{cases}\]
with the observation that 
\begin{equation}\label{eqn: RCF functional}T(t) = F_0 \circ F_1 ^{a(t)-1} (t).\end{equation}
Note that $a(t)-1=\min(n\in\mathbb{N}_0 : T^n(t)\in [\frac{1}{2},1])$ i.e. it is first hitting time of $x$ to the interval $[\frac{1}{2},1]$ (possibly zero) and $T$ can be understood as the composition of $F_0$ and the induced transformation on this interval (where a return time of zero is allowed).

A generalization of regular continued fractions, $N$-continued fractions, allow for all numerators to be some fixed $N \in \mathbb{Z}\setminus\{0\}$. These were first studied in  \cite{AW11,BGRKWY208} and later in many other papers such as \cite{CK,DO18,DKW13,JKN,KL} . Most papers study positive values of $N$ and in the remainder of this section we will also take $N>0$. For negative $N$, there is a relation with the backward continued fractions in Section \ref{sec:other trees}. One of the key differences with regular continued fractions is that there is not a unique continued fraction for every $t\in(0,1)$. For some studied algorithms generating $N$-continued fractions it is proven that there are quadratic irrationals without periodic expansion (and even rationals with aperiodic expansions), see \cite{KL}. For the greedy $N$-continued fractions, which we have here, there is strong numeric evidence that there are quadratic irrationals with an aperiodic expansion (see \cite{DKW13}), though this problem is still open. The greedy $N$-continued fractions can be generated with a similar map as the Gauss map namely by 
\[ T_N(t) = \begin{cases}
   \left(\frac{N}{t}-N\right) - \floor{\frac{N}{t}-N} & t \neq 0\\
     0 & t=0,
\end{cases}\]
which generates
\[ t= \cfrac{N}{N+(a_1-1)+\cfrac{N}{N+(a_2-1)+\cfrac{N}{N+(a_3-1)+\ddots}}},\]
where now $a_i=a(N,t)= \floor{N/T_N^{i-1}(t)-N}+1$. Then $T_N(x)$ is a sped-up version of the associated $N$-Farey map
\[F_N(t) = \begin{cases}
    F_{N,0}(t) = \frac{N(1-t)}{t} &  t \geq \frac{N}{N+1} \\
    F_{N,1}(t) = \frac{Nt}{N-t} & t \leq \frac{N}{N+1}
\end{cases}\]
with the observation that
\begin{equation}\label{eqn: N-CF functional}T_N(t) = F_{N,0} \circ F_{N,1}^{a(N,t)-1}(t).\end{equation}
See Figure \ref{fig:Fareymap} (right) for an example of $N=2$. It is important to point out that this definition of $a_i=a(N,t)$ is not typical: more typical, e.g. in the above references for N-continued fractions, is that $a(N,t)=\floor{N/t}$ and $T_N(t)$ is equivalent but without the `$-N$' both inside and outside the integer part function. Our definition here differs only by an integer constant for $a$, is identical for $T_N$, gives a similar presentation between \Cref{eqn: RCF functional} and (\ref{eqn: N-CF functional}), and is what we will now generalize. 

Note that from a dynamical point of view, there is no necessity of taking $N\in\mathbb{Z}^+$ in the definition of the map $F_N(t)$. To this end, let us define a map $F_x:[0,1]\rightarrow[0,1]$ with $x\in(0,\infty)$\footnote{note that for other values of $x$, we have that $\frac{x}{x+1}$ is not in $[0,1]$  and the map cannot be used to generate continued fractions.} as
\begin{equation}\label{eq:F_x}
  F_x(t)=
\begin{cases}
F_{x,0}(t)=\frac{x(1-t)}{t}   & t \geq \frac{x}{x+1}\\[1ex]
F_{x,1}(t)=\frac{xt}{x-t}, & t \leq \frac{x}{x+1},
\end{cases}  
\end{equation}
see Figure \ref{fig:Fareymap} (left) for $x=1/3$.
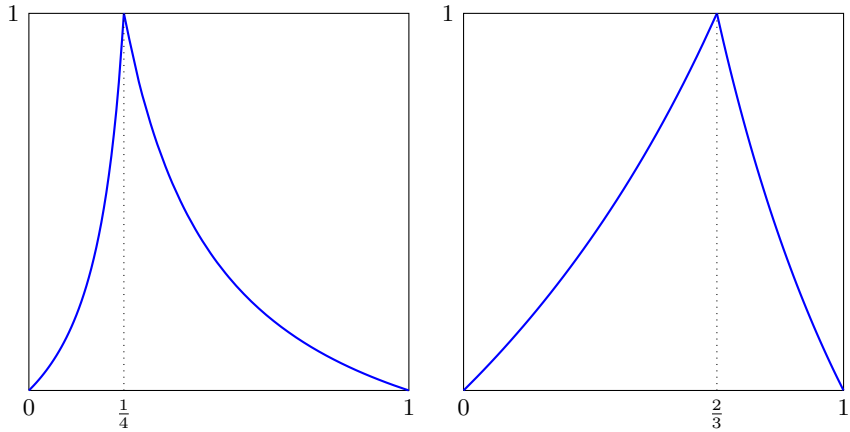
\begin{figure}[ht]
		\centering
		
		\subfigure{\begin{tikzpicture}[scale=5]
				\draw[white] (-0.25,0)--(1,0);
				\draw(0,0)node[below]{\small $0$}--(1,0)node[below]{\small $1$}--(1,1)--(0,1)node[left]{\small $1$}--(0,0);

				\draw[thick, blue, smooth, samples =20, domain=1/4:1] plot(\x,{1/(3*\x)-1/3});
				\draw[thick,blue, smooth, samples =20, domain=0:1/4] plot(\x,{\x/(1-3*\x)});

				\draw[dotted](1/4,0)node[below]{\small $\frac{1}{4}$}--(1/4,1);
				
		\end{tikzpicture}}
        \subfigure{\begin{tikzpicture}[scale=5]
				\draw(0,0)node[below]{\small $0$}--(1,0)node[below]{\small $1$}--(1,1)--(0,1)node[left]{\small $1$}--(0,0);

				\draw[thick, blue, smooth, samples =20, domain=2/3:1] plot(\x,{2/\x-2});
				\draw[thick,blue, smooth, samples =20, domain=0:2/3] plot(\x,{2*\x/(2-\x)});

				\draw[dotted](2/3,0)node[below]{\small $\frac{2}{3}$}--(2/3,1);
				
		\end{tikzpicture}}
		
		\caption{The map $F_x$ for $x=1/3$ on the left and $x=2$ on the right. \hspace{5.5em} }
		\label{fig:Fareymap}
	\end{figure}

The sped-up version of this map will be the induced transformation on $(\frac{x}{x+1},1]$. This leads us to the following Gauss like map.
Let $x\in (0,\infty)$ and define the map 
$ T_x\colon[0,1]\to[0,1]$ as
\begin{equation}\label{eq:T_x}
  T_x(t)=
\begin{cases}
\left(\dfrac{x}{t}-x\right)- \floor{\dfrac{x}{t}-x}
 
& t\neq0,\\[1ex]
0 & t=0.
\end{cases}  
\end{equation}
See Figure \ref{fig:Tx} for the graph for $x=2$, $5/2$ and $\pi$.
\begin{figure}[ht]
\centering
\subfigure[$T_{2}$]{\begin{tikzpicture}[scale=2.2]
\draw[white] (-.6,0)--(1.5,0);
\draw(0,0)node[below]{\tiny $0$}--(1,0)node[below]{\tiny $1$}--(1,1)--(0,1)node[left]{\tiny $1$}--(0,0);

\foreach \n in {2,...,7}{%
	 \pgfmathsetmacro{\xval}{2/(\n+1)};
    \draw[dotted] (\xval,0)--(\xval,1);
     \draw[thick, purple!50!black, smooth, samples =20, domain=2/(\n+1):2/\n] plot(\x,{2 / \x -\n});
}
\draw(2/3,0)node[below]{\tiny $\frac{2}{3}$};
\draw(2/4,0)node[below]{\tiny $\frac{2}{4}$};
\draw(2/5,0)node[below]{\tiny $\frac{2}{5}$};

\fill[purple!50!black] (0.05, 0.5) circle (0.5pt); 
\fill[purple!50!black] (0.125, 0.5) circle (0.5pt); 
\fill[purple!50!black] (0.2, 0.5) circle (0.5pt); 

\end{tikzpicture}}
\subfigure[$T_{5/2}$]{\begin{tikzpicture}[scale=2.2]
\draw[white] (-.6,0)--(1.5,0);
\draw(0,0)node[below]{\tiny $0$}--(1,0)node[below]{\tiny $1$}--(1,1)--(0,1)node[left]{\tiny $1$}--(0,0);

\draw(2/3,0)node[below]{\tiny $\frac{5}{7}$};
\draw(2/4,0)node[below]{\tiny $\frac{5}{9}$};
\draw(2/5,0)node[below]{\tiny $\frac{5}{11}$};

\foreach \n in {2,...,7}{%
	 \pgfmathsetmacro{\xval}{2.5/(\n+1.5)};
    \draw[dotted] (\xval,0)--(\xval,1);
     \draw[thick, purple!50!black, smooth, samples =20, domain=2.5/(\n+1.5):2.5/(\n+0.5)] plot(\x,{2.5 / \x -0.5 -\n});
}

\fill[purple!50!black] (0.05, 0.5) circle (0.5pt); 
\fill[purple!50!black] (0.125, 0.5) circle (0.5pt); 
\fill[purple!50!black] (0.2, 0.5) circle (0.5pt); 

\end{tikzpicture}}
\subfigure[$T_\pi$]{\begin{tikzpicture}[scale=2.2]
\draw[white] (-.6,0)--(1.5,0);
\draw(0,0)node[below]{\tiny $0$}--(1,0)node[below]{\tiny $1$}--(1,1)--(0,1)node[left]{\tiny $1$}--(0,0);

\draw({3.1415/(4.1415)},0)node[below]{\tiny $\frac{\pi}{\pi+1}$};

\foreach \n in {3,...,9}{%
	 \pgfmathsetmacro{\xval}{pi/(\n+1+(pi - floor(pi))))};
    \draw[dotted] (\xval,0)--(\xval,1);
     \draw[thick, purple!50!black, smooth, samples =20, domain=pi/(\n+1+(pi - floor(pi))):pi/(\n +(pi - floor(pi)))] plot(\x,{pi / \x -(pi - floor(pi)) -\n});
}
\fill[purple!50!black] (0.05, 0.5) circle (0.5pt); 
\fill[purple!50!black] (0.125, 0.5) circle (0.5pt); 
\fill[purple!50!black] (0.2, 0.5) circle (0.5pt); 

\end{tikzpicture}}

\caption{The map $T_x$ for $x=2$, $5/2$ and $\pi$.}
\label{fig:Tx}
\end{figure}
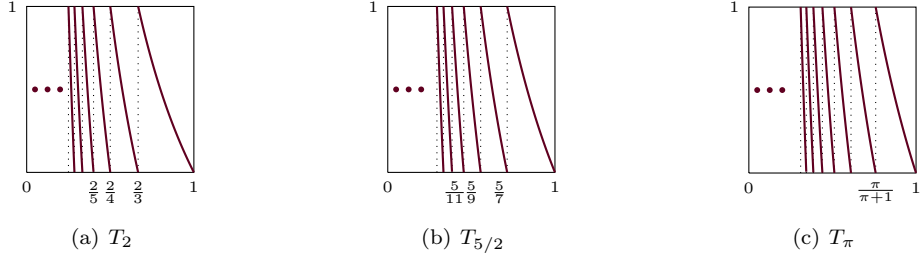

Furthermore, set
$
a(t)=\bigl\lfloor\tfrac{x}{t} -x\bigr\rfloor+1
$ and $
a_n(t)=a\bigl(T_x^{\,n-1}(t)\bigr) $ for $(n\ge2),$ so that $T_x(t) = x/t-x-(a_1(t)-1)$, and just as with $T$ and $T_N$ we have
\[T_x(t) = F_{x,0}\circ F_{x,1}^{a(x,t)-1}(t).\] 

The `$+1$' in the definition of $a_1$ is to ensure that $a_1(t) \in \mathbb{Z}^+$; as previously remarked, $\floor{x/t-x}$ yields the number of applications of $F_{x,1}$ to orbit $t$ into $(x/(x+1),1]$, and then a single extra application of $F_{x,0}$ completes the `sped-up' map. Note also that for $x=1$ one recovers the regular continued fraction algorithm. Then for $t\in(0,1)$ one obtains 
\[ t = \frac{x}{{x}+(a_1(t)-1)+T_x(t)}\]
from which we obtain the $x$–continued fraction expansion
\[
t
=\cfrac{x}{x+(a_1-1)+\cfrac{x}{x+ (a_2-1)+\cfrac{x}{x+(a_3-1)+\ddots}}}
:=\bigl[a_1,a_2,a_3,\dots\bigr]_x,
\]
which is finite if and only if there exists an $n$ such that $T_x^n(t)=0$. 

When we remove the `$-x$' outside as well as inside the integer part in \eqref{eq:T_x} then we get the generalization of the regular continued fractions as in \cite{GS17,M20}. Their approach has the advantage of having natural numbers as the digits:
\[t = \cfrac{x}{a_1+\cfrac{x}{a_2+\ddots}}\] where the $a_i$ are positive integers. On the other hand, the non-full branch they have close to $1$ makes that generalization difficult to study. In our setting all branches are full and many properties of the regular continued fractions are preserved, either exactly or in an analogous way. We remark that for $x\in(0,1)$ the analogy is not always apparent as the map $T_x$ in that  case is not uniformly expanding.  The following proposition sums up some of the properties for $x\geq 1$.

\begin{proposition}
For $x\geq 1$, we have the following properties 
  \begin{itemize}
  \item The dynamical system $ \bigl(T_x,\;[0,1],\;\mathcal{B},\;\mu_x\bigr)$,  where $\mathcal{B}$ is the Borel $\sigma$–algebra on $[0,1]$ and $\mu_x$ is the invariant probability measure
  \[
    \mu_x(A)
    \;=\;
    \frac{1}{\ln \bigl(\tfrac{x+1}{x}\bigr)}
    \int_{A} \frac{1}{x + t}\,dt,
    \qquad A\in  \mathcal{B},
  \]
  is exact and hence ergodic.

  \item The entropy of this system is
  \[
    h\bigl(T_x\bigr)
    \;=\;
    \frac{
      \dfrac{\pi^2}{3} \;+\; 2\,Li_2(x+1)\;+\;\ln(x+1)\,\ln(x)
    }{
      \ln\bigl(\tfrac{x+1}{x}\bigr)
    },
  \]
  where $Li_2$ denotes the dilogarithm function, i.e, $Li_2(x)=\int_0^x \frac{\log(t)}{1-t} dt$.
  \item The natural extension is given by $([0,1)\times[0,1],\mathcal{B},\rho_x,\mathcal{T}_x)$ where $\mathcal{B}$ is the Borel $\sigma$-algebra on $[0,1)\times[0,1]$, $\rho_x$ the invariant measure with density $\frac{1}{\log(\frac{x+1}{x})}\frac{x}{(x+ty)^2}$  and $\mathcal{T}_x(t,y)=(T_x(t),\frac{x}{d_1(t)+y})$.
\end{itemize}
\end{proposition}
These properties might or might not hold for $x\in(0,1)$ but are not needed for our proofs; they follow from \cite{DKW13}. Other properties like Khintchines constants, a Doeplin-Lenstra type theorem, and other Diophantine properties are probably within reach but out of the scope of this paper.\\

Referring back to \eqref{eq:F_x}, we find two inverses for each $t \in [0,1]$ which we denote as
\begin{equation}\label{eq: first Phi definition}
\Phi_0(t) = \frac{x}{x+t}, \qquad \Phi_1(t) = \frac{xt}{x+t}.
\end{equation}
For $x=1$, i.e. $F_x$ is the standard Farey map, the orbit of any rational $t\in (0,1)$ will be finite, eventually reaching $1/2$, the critical point of $F_1$ which is then mapped to one, and then the fixed point zero. Alternately, the tree generated by $\Phi_0$ and $\Phi_1$ rooted at $1/2$ will generate all rational numbers in $(0,1)$. So we clarify our notation to show how $\Phi_0$ and $\Phi_1$ may be used to generate a tree rooted at $x/(x+1)$, the critical point of $F_x$:
\[
\Phi_0\left(f(x)\right) = \frac{x}{x+f(x)}, \qquad \Phi_1\left(f(x)\right)=\frac{xf(x)}{x+f(x)}
\]
We now set the domain of $\Phi_0,\Phi_1$ to be \textit{arbitrary functions $f(x)$} observing that if we begin with a root of $x/(x+1)$ we generate a tree of rational functions. In this setting it is convenient to set $f(x) = p(x)/q(x)$. Depending on whether $f(0)=0$ or not, the functions $\Phi_0(f)$ and $\Phi_1(f)$ may have a removable discontinuity at $x=0$. We elect to remove this discontinuity when possible, which leads directly to \Cref{eqn: Phi} as our definition of the maps $\Phi_0$, $\Phi_1$.

The rational functions in the tree (\Cref{fig:FareyPoly}) generated by \Cref{eqn: Phi} may be directly represented with continued fraction notation. We introduced this notation in \Cref{eqn: vertex labels}, but we now justify the connection, that is, for a given $x$, if $t=[a_1,a_2,\ldots,a_k]_x$ then $t=\Phi_1^{a_1-1} \circ \Phi_0 \circ \Phi_1^{a_2-1} \circ \Phi_0 \circ \cdots \circ \Phi_0 \circ \Phi_1^{a_k-1} \left( \frac{1}{1}\right)$.
\begin{lemma}
    \label{lem: CF notation for vertices}
    For any $a \geq 1$, we have 
    \[ \Phi_1^{a-1} \left( \frac{1}{1}\right)=\frac{x}{x+a-1}.\]
    For a rational function $p(x)/q(x)$, we have
    \[ \Phi_1^{a-1} \circ \Phi_0 \left( \frac{p(x)}{q(x)}\right) = \frac{x}{x+(a-1)+\frac{p(x)}{q(x)}}.\]
    \begin{proof}
        Both claims will utilize the fact that for $a \geq 2$
        \[ \Phi_1^{a-1} \left( \frac{p(x)}{q(x)}\right) = \frac{x^{a-1}p(x)}{x^{a-2} \left(xq(x) + (a-1) p(x) \right)} = \frac{xp(x)}{xq(x) + (a-1)p(x)},\]
        where $p(0)=0$ may lead to further removable discontinuities at $x=0$, but which will not affect values at other $x$. Also, the above holds trivially for $a=1$ (again ignoring a removable discontinuity at $x=0$).

        The first claim is therefore proved, and for the second we apply the above to $\Phi_0$:
        \begin{align*}
        \Phi_1^{a-1} \circ \Phi_0 \left( \frac{p(x)}{q(x)}\right) &= \Phi_1^{a-1} \left( \frac{xq(x)}{xq(x)+p(x)}\right)\\
        \\
        &=\frac{x^2q(x)}{x(xq(x)+p(x))+(a-1)xq(x)}\\
        \\
        &=\frac{x}{x+(a-1) + \frac{p(x)}{q(x)}}.\qedhere
        \end{align*}
    \end{proof}
\end{lemma}

In the following section we prove the necessary properties of the continued fractions generated by $T_x$ in order to prove Theorem \ref{th:cfs}. 

\subsection{Continued Fraction Results}\label{sec:cfresults}
We will first prove convergence of the continued fractions for $x\in(0,1)$. For $x \geq 1$ the proof is analogous to the proof of convergence for proper continued fractions, see \cite{L40}. 
The $n$th convergent of $t\in(0,1)$ is given by
\[
c_n(t)=\frac{p_n}{q_n}=\bigl[0\,;\,a_1(t),a_2(t),\dots,a_n(t)\bigr]_x.
\]
Where $p_n$ and $q_n$ satisfy the relations
\begin{equation}\label{eq:rec for pn and qn}
    \begin{aligned}
        p_{-1}=1,&\, p_0=0, & p_n=(x+a_n-1)p_{n-1}+xp_{n-2},~\text{for } n\geq 1,\\
    q_{-1}=0, &\, q_0=1, & q_n=(x+a_n-1)q_{n-1}+xq_{n-2},~\text{for } n\geq 1.
    \end{aligned}
\end{equation}

\begin{proposition}\label{prop: convergence for 0<x<1}
    For $x\in(0,1)$ the continued fractions generated by the map $T_x$ converge as well.
\end{proposition}
\begin{proof}
    Just as for $x \geq 1$, we have the following inequality:
\begin{equation}
    \left| t-\frac{p_n}{q_n} \right|<\frac{x^n}{q_n^2}.
\end{equation}
For $x\in(0,1)$ the $q_n$ are not necessarily increasing but we will show that $\lim_{n\to \infty}\frac{x^n}{q_n^2}=0 $. Note that the least value of $q_n$ achieved by taking $a_n=1$ for all $n$ so it suffices to prove convergence for that case. 
In order to do so, we show that for $n\geq3$
\begin{align}
    q_n&>\frac{n}{2}x^{\frac{n+2}{2}}+x^{\frac{n}{2}}& n \text{ even,}\label{induction:even}\\
    q_n&>\frac{n+1}{2}x^{\frac{n+1}{2}}& n \text{ odd.}\label{induction:odd}
\end{align}
This is done by induction. Note that for $n=3$ we have $q_3=x^3+2x^2>2x^2$ and $q_4=x^4+3x^3+x^2>2x^3+x^2$. So suppose $n$ is even. Then under our inductive hypotheses we find
\[
q_{n+1}=x(q_n+q_{n-1})>x\left(\frac{n}{2}x^{\frac{n+2}{2}}+x^{\frac{n}{2}}+\frac{n}{2}x^{\frac{n}{2}}\right)=\frac{n}{2}x^{\frac{n+4}{2}}+\frac{n+2}{2}x^{\frac{n+2}{2}}>\frac{n+2}{2}x^{\frac{n+2}{2}}
\]
and for $n$ odd:
\[
q_{n+1}=x(q_n+q_{n-1})>x\left(\frac{n+1}{2}x^{\frac{n+1}{2}} +\frac{n-1}{2}x^{\frac{n+1}{2}}+x^{\frac{n-1}{2}}  \right)=nx^{\frac{n+3}{2}}+x^{\frac{n+1}{2}}.
\]
Using \eqref{induction:even} for $n$ even we find
\[
\frac{x^n}{q_n^2}<\frac{x^n}{(\frac{n}{2}x^{\frac{n+2}{2}}+x^{\frac{n}{2}})^2}=\frac{x^n}{\frac{n^2}{4}x^{n+2}+nx^{n+1}+x^n}=\frac{1}{\frac{n^2}{4}x^{2}+nx^{1}+1}.
\]
Using \eqref{induction:odd} for $n$ odd we find
\[
\frac{x^n}{q_n^2}<\frac{x^n}{(\frac{n+1}{2}x^{\frac{n+1}{2}})^2}=\frac{1}{\frac{(n+1)^2}{4}x}.
\]
Of course between both cases we get that $|t-p_n/q_n|$ converges to zero as $n \rightarrow \infty$.
\end{proof}
It is of interest to characterize for which $x$ all rationals in $(0,1)$ have a finite expansion. 

\begin{proposition}\label{prop:finite}
 For $x\in\mathbb{Q}_{>0}$,   all rationals in $(0,1)$ have a finite $x$-expansion if and only if $x\in\mathbb{N}$.  
\end{proposition}
\begin{proof}
Let $\frac{t_0}{s_0}\in(0,1)$ coprime and set $T_x^n(\frac{t_0}{s_0})=\frac{t_n}{s_n}$ (so long as defined).\\ For $x\in \mathbb{N}$ we have the following.
\[
T_x\left(\frac{t_n}{s_n}\right)=\frac{s_n x-at_n}{t_n}=\frac{t_{n+1}}{s_{n+1}}.
\]

Since $\frac{t_{n+1}}{s_{n+1}}\in (0,1)$ we find $t_{n+1}<s_{n+1}=t_n$ even without dividing out common divisors. Therefore the numerators decrease when applying $T_x$ and will eventually be zero.\\
For $x\in\mathbb{Q}\backslash \mathbb{N}$ we write $x=p/q$ with $p,q$ coprime and $\hat{p}\equiv p\mod q$ with $\hat{p}<q$.
\[
T_x\left(\frac{t_n}{s_n}\right)=\frac{s_n p-at_n q -\hat{p}t_n}{t_nq}=\frac{t_{n+1}}{s_{n+1}}.
\]
Now for $s_n\equiv 0 \mod q$ we find that $t_{n+1} \mod q \equiv - \hat{p}t_n \mod q \not\equiv 0 \mod q $ and $s_{n+1}\equiv 0 \mod q$. In other words, if we start with $s_0$ divisible by $q$ and $t_0$ not divisible by $q$, the numerators of any point in the orbit will never be $0 \mod q$ and therefore never can be $0$. (On the other hand the denominators will always be divisible by $q$.)
\end{proof}

If $x \notin \mathbb{Q}$, then for any rational $t_0/s_0$, we have
\[ T_x\left( \frac{t_0}{s_0}\right) = \left(\frac{x}{t_0/s_0} -x\right)- \floor{\frac{x}{t_0/s_0}-x}= \frac{(s_0-t_0) x}{t_0}-a\]
which is certainly irrational. On the other hand, pre-images of rational numbers can be irrational too. Therefore, to give a general statement about whether all rationals have a finite expansion for generic values of irrational $x$ seems rather non-trivial. Note that it can happen for specific irrational values of $x$ that there are finite expansions for rational numbers. For example for $x=\sqrt{2}$ we find 
\[[2,3]_x = \cfrac{\sqrt{2}}{\sqrt{2}+1+\cfrac{\sqrt{2}}{\sqrt{2}+2}}=\frac{1}{2}\]
and similarly for $x=1+\frac{1}{2}\sqrt{14}$ we find $\frac{1}{3}=[6,2]_x$. This classification problem seems to be similar in nature as for quadratic irrationals and rationals for certain $N$-continued fraction algorithms, see~\cite{KL}. If a number has a finite expansion, one could argue it is `perfectly approximable.' For algebraic values of $x$, it could be interesting to investigate which numbers are perfectly approximable, which ones are well approximable (having very high digits) and which ones are badly approximable (having bounded digits). These questions are beyond the scope of this article.

\begin{proof}[Proof of Theorem \ref{th:cfs}]
The fact that we find a subset of $[0,1]$ for all $x\in(0,\infty)$ immediately follows from the relation with continued fractions. That the set is dense follows from the convergence of those continued fractions. Finally, for $x\in\mathbb{Q}_{\geq 0}$ the appearance of every rational exactly once in the tree if and only if $x\in \mathbb{Z}^+$ follows from Proposition \ref{prop:finite}. 
\end{proof}

\subsection{Density Result}\label{sec:dense}
In this section we present collected results regarding our tree that are not directly related to traditional continued fractions. We will generally write $p(x)/q(x)=p/q$ to simplify notation, but recall throughout that $p,q$ are polynomials and not integers. 

\begin{proposition}\label{prop: elementary vertex properties without proof}
    At every  vertex $[a_1,\ldots,a_k]_x=p/q$, we have:
    \begin{itemize}
        \item $p,q$ are polynomials of the same degree.
        \item All coefficients of each are positive integers, except the constant term of $p(x)$ which may be zero.
        \item Term-by-term, the coefficients of $q$ are strictly larger than those of $p$, except for the leading coefficients which are both one.
    \end{itemize}
    \begin{proof}
        All properties are preserved by both $\Phi_0$ and $\Phi_1$, and are true for the root vertex $p(x)=x$, $q(x)=x+1$.
    \end{proof}
\end{proposition}
\begin{proposition}\label{prop: increasing functions}
    For every vertex $v=[a_1,\ldots,a_k]_x$, we have both
    \begin{align}
       v'(x)&> 0 &\textrm{(wherever defined)}\\
        xv'(x) &\leq v(x) & \textrm{(with equality if and only if $x=0$ and $v(0)=0$)}
    \end{align}
    \begin{proof}
    We induct on $k$, the length of the expansion of $v$. For $k=1$ we have 
    \[v=\frac{x}{x+a_1-1}\]
    (where $a_1 \geq 2$) for which both claims may be trivially verified. Now assume both properties hold for $v$, and for some $a \in \mathbb{Z}^+$ set 
    \[V = [a,v]_x = \begin{cases}\frac{x}{x+a-1+v} & a \geq 2 \textrm{ or } v(0) \neq 0\\ \frac{1}{1+v/x} & a=1 \textrm{ and } v(0)=0.       
    \end{cases}\]
    Note that by \Cref{prop: elementary vertex properties without proof} $x=0$ is at most a simple root of $v$, so the second case may be properly defined even for $x=0$, and if we establish $v'>0$, then $xv'$ also has at most a simple root at $x=0$.
    In the first case, we compute
    \begin{align*}
    V'&=\frac{a-1+v-xv'}{(x+a-1+v)^2}.\\
    \intertext{Since we have assumed $xv'<v$, the first claim now follows.}
    xV'-V&= \frac{x((a-1)+v-xv')-x(x+a-1+v)}{(x+a-1+v)^2}\\
    &=-x^2\frac{v'+1}{(x+a-1+v)^2}.\\
    \end{align*}
    Which is negative unless $x=0$, in which case it follows that $V(0)=0$ as well.

    In the second case the same proof applies for all $x \neq 0$; only at $x=0$ is there any ambiguity. Then we may set $v=(x^n+\cdots+a_1x)/(x^n+\cdots+b_1x+b_0)$ where $a_1,b_0$ are positive integers, and simply from the definition of the derivative compute that $v'(0)=a_1/b_0 >0$ and trivially $xv'=v$ as both are zero.
    \end{proof}
\end{proposition}
\begin{proposition}\label{prop: root ordering}
Let $v=[a_1,\ldots,a_k]_x=p/q$ be any vertex, and let $p,q$ be of degree $n$. Then we may write
\[v(x) = \prod_{i=1}^n \left(\frac{x-r_i}{x-s_i}\right),\]
where $s_n \leq -1$, and all $r_i$, $s_i$ may be ordered as
\[s_n < r_n < \cdots <s_1 <r_1 \leq 0.\]
That is: $p,q$ are always fully-factorable over $\mathbb{R}$ with all roots of multiplicity one, share no roots, have no positive roots, between two roots of one there is always a root of the other, and the largest root of the two is that of $p$.
    \begin{proof}
    We induct on $k$. For $k=1$ we have $v=x/(x+a_1-1)$ with $a_1 \geq 2$ and all properties follow immediately. So suppose $v$ satisfies all claimed properties, we write
    \[v = \prod_{i=1}^n \left( \frac{x-r_i}{x-s_i}\right),\]
    and set
    \[V = [a,v]_x = \frac{x}{x+(a-1)+v} = \begin{cases}
        \frac{xq}{(x+a-1)q+p} & (a \geq 2 \textrm{ or } p(0)\neq 0)\\ \frac{q}{q+p/x} & (a=1 \textrm{ and }p(0)=0)
    \end{cases}\]
    In the first case, note that the roots of $xq(x)$ (the numerator of $V$) will be exactly $s_n,s_{n-1},\ldots,s_1,0$, while in the second case the roots of the numerator of $V$ will be the same, but without $0$.

    In both cases, roots of the denominator of $V$ are given by the solutions to $v(x) = -(x+a-1)$, i.e. 
    \[x + v(x) = -(a-1).\]
    The function $x+v(x)$ is strictly increasing (\Cref{prop: increasing functions}), so there is a unique solution in each interval $(s_i,s_{i-1})$. As $x \rightarrow -\infty$, we have $v(x)$ decreasing to one, so there is exactly one other solution on the interval $(-\infty,s_n)$. In the second case the proof is now complete.

    In the first case as $p(0)>-(a-1)$ (as either $a \geq 2$ or $p(0)>0$) we find one more solution in the interval $(s_1,0)$, completing the proof.
    \end{proof}
\end{proposition}

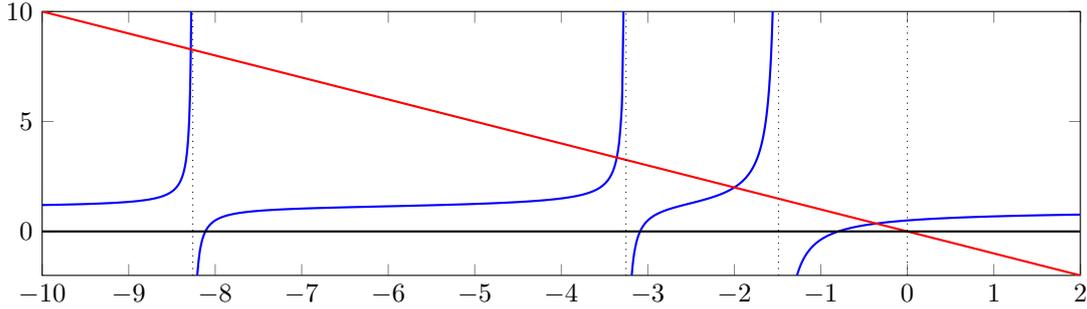
\begin{figure}[h]
\centering
\begin{tikzpicture}[domain=-10:2]
\begin{axis}[width=.9 \linewidth, height= .3 \linewidth,
    xmin = -10, xmax = 2,
    ymin = -2, ymax = 10]
    \addplot[
        domain = -10:-8.27,
        samples = 200,
        smooth,
        thick,
        blue,
    ] {(x^3+12*x^2+34*x+20)/(x^3+13*x^2+44*x+40)};
        \addplot[
        domain = -8.24:-3.27,
        samples = 200,
        smooth,
        thick,
        blue,
    ] {(x^3+12*x^2+34*x+20)/(x^3+13*x^2+44*x+40)};
        \addplot[
        domain = -3.25:-1.5,
        samples = 200,
        smooth,
        thick,
        blue,
    ] {(x^3+12*x^2+34*x+20)/(x^3+13*x^2+44*x+40)};
        \addplot[
        domain = -1.47:2,
        samples = 200,
        smooth,
        thick,
        blue,
    ] {(x^3+12*x^2+34*x+20)/(x^3+13*x^2+44*x+40)};
    \draw[red, thick](-10,10)--(2,-2);
    \draw[black,dotted](-8.259,-2)--(-8.259,10);
    \draw[black,dotted](-3.252,-2)--(-3.252,10);
    \draw[black,dotted](-1.489,-2)--(-1.489,10);
    \draw[black,dotted](0,-2)--(0,10);
    \draw[black,thick](-10,0)--(2,0);
\end{axis}
\end{tikzpicture}
\caption{For the vertex $[1,2,6,4]_x$, given by the composition $\Phi_0 \Phi_1 \Phi_0 \Phi_1^5 \Phi_0 \Phi_1^3(\frac{1}{1})$, yielding the function $(x^3+12x^2+34x+20)/(x^3+13x^2+44x+40)$. The function is graphed, along with $y=-x$; intersections will be roots of the denominator of either child vertex. As $p(0)\neq 0$, either child vertex will also have $x=0$ as a root. For either $\Phi_0$ or $\Phi_1$, the `alternating sequence' of roots of numerator and denominator is preserved at the child vertex.}
\label{fig: rational plot}
\end{figure}

We use \Cref{prop: root ordering} to establish four separate steps which will complete the proof of \Cref{th:dense}. We will need to understand how the roots and poles of $p/q$ align in contrast to those of both $\Phi_0(p/q)$ and $\Phi_1(p/q)$, both of which then depend on $p(0)=0$ versus $p(0) \neq 0$. See \Cref{fig: all the subfigures} for a collection of figures illustrating these cases.
\begin{figure}[bth]
    \begin{center}
    \subfigure[For $p/q$ with $p(0)= 0$, considering $\Phi_1$, set $\hat{p}=p$ and $\hat{q}=q+p/x$.]{
\begin{tikzpicture}[scale=.6]
    \draw[<-] (-1,0)--(5,0);
    \draw[style=dotted] (5,0)--(8,0);
    \draw (8,0)--(14,0);
    \draw[style=dotted] (14,0)--(17,0);
    \draw (17,0)--(22,0);
    \draw[color=purple] (4,0)--(4,1)node[above]{$r_n=\hat{r}_n$};
    \draw[color=purple] (9,0)--(9,1)node[above]{$r_{i+1}=\hat{r}_{i+1}$};
    \draw[color=purple] (13,0)--(13,1)node[above]{$r_i=\hat{r}_i$};
    \draw[color=purple] (18,0)--(18,1)node[above]{$r_2=\hat{r}_2$};
    \draw[color=purple] (22,0)--(22,1)node[above]{$r_1=\hat{r}_1=0$};
    \draw[color=red] (1,0)--(1,-2)node[below]{$\hat{s}_n$};
    \draw[color=red] (10,0)--(10,-2)node[below]{$\hat{s}_i$};
    \draw[color=red] (19,0)--(19,-2)node[below]{$\hat{s}_1$};
    \draw[color=blue] (2,0)--(2,-1)node[below]{$s_n$};
    \draw[color=blue] (11,0)--(11,-1)node[below]{$s_i$};
    \draw[color=blue] (20,0)--(20,-1)node[below]{$s_1$};
    \draw (-2,-3) rectangle (24,3);
\end{tikzpicture}
\label{subfig: num zero phi 1}
}
\subfigure[For $p/q$ with $p(0)\neq0$, considering $\Phi_1$, set $\hat{p}=xp$ and $\hat{q}=xq+p$. The situation is exactly identical to \Cref{subfig: num zero phi 1} but for the addition of one more root of each at/near $x=0$ and shifting of indices.]
{\begin{tikzpicture}[scale=.6]
    \draw[<-] (-1,0)--(5,0);
    \draw[style=dotted] (5,0)--(8,0);
    \draw (8,0)--(14,0);
    \draw[style=dotted] (14,0)--(17,0);
    \draw (17,0)--(22,0);
    \draw[color=purple] (4,0)--(4,1)node[above]{$r_n=\hat{r}_{n+1}$};
    \draw[color=purple] (9,0)--(9,1)node[above]{$r_{i}=\hat{r}_{i+1}$};
    \draw[color=purple] (13,0)--(13,1)node[above]{$r_{i-1}=\hat{r}_i$};
    \draw[color=purple] (18,0)--(18,1)node[above]{$r_1=\hat{r}_2$};
    \draw[color=purple] (22,0)--(22,1)node[above]{$\hat{r}_1=0$};
    \draw[color=red] (1,0)--(1,-2)node[below]{$\hat{s}_{n+1}$};
    \draw[color=red] (10,0)--(10,-2)node[below]{$\hat{s}_i$};
    \draw[color=red] (19,0)--(19,-2)node[below]{$\hat{s}_1$};
    \draw[color=blue] (2,0)--(2,-1)node[below]{$s_n$};
    \draw[color=blue] (11,0)--(11,-1)node[below]{$s_{i-1}$};
    \draw (-2,-3) rectangle (24,3);
\end{tikzpicture}
\label{subfig: num nonzero phi 1}
}
\subfigure[For $p/q$ with $p(0)=0$, considering $\Phi_0$, set $\hat{p}=q$ and $\hat{q}=q+p/x$.]{
\begin{tikzpicture}[scale=.6]
    \draw[<-] (-1,0)--(4,0);
    \draw[style=dotted] (4,0)--(7,0);
    \draw (7,0)--(14,0);
    \draw[style=dotted] (14,0)--(17,0);
    \draw (17,0)--(22,0);
    \draw[color=blue] (3,0)--(3,1)node[above]{$r_n$};
    \draw[color=blue] (10,0)--(10,1)node[above]{$r_{i}$};
    \draw[color=blue] (18,0)--(18,1)node[above]{$r_2$};
    \draw[color=blue] (22,0)--(22,1)node[above]{$r_1=0$};
    \draw[color=red] (2,0)--(2,2)node[above]{$\hat{r}_n$};
    \draw[color=red] (8,0)--(8,2)node[above]{$\hat{r}_{i}$};
    \draw[color=red] (13,0)--(13,2)node[above]{$\hat{r}_{i-1}$};
    \draw[color=red] (20,0)--(20,2)node[above]{$\hat{r}_1$};
    \draw[color=blue] (2,0)--(2,-1)node[below]{$s_n$};
    \draw[color=blue] (8,0)--(8,-1)node[below]{$s_{i}$};
    \draw[color=blue] (13,0)--(13,-1)node[below]{$s_{i-1}$};
    \draw[color=blue] (20,0)--(20,-1)node[below]{$s_1$};
    \draw[color=red] (1,0)--(1,-2)node[below]{$\hat{s}_{n}$};
    \draw[color=red] (11,0)--(11,-2)node[below]{$\hat{s}_{i}$};
    \draw[color=red] (19,0)--(19,-2)node[below]{$\hat{s}_1$};
    \draw (-2,-3) rectangle (24,3);
\end{tikzpicture}
\label{subfig: num zero phi 0}
}
\subfigure[For $p/q$ with $p(0)\neq 0$, considering $\Phi_0$, set $\hat{p}=xq$ and $\hat{q}=xq+p$. The situation is identical to \Cref{subfig: num zero phi 0} except for the addition of one root of each at/near $x=0$ and shifting of indices.]{
\begin{tikzpicture}[scale=.6]
    \draw[<-] (-1,0)--(4,0);
    \draw[style=dotted] (4,0)--(7,0);
    \draw (7,0)--(14,0);
    \draw[style=dotted] (14,0)--(17,0);
    \draw (17,0)--(22,0);
    \draw[color=blue] (3,0)--(3,1)node[above]{$r_n$};
    \draw[color=blue] (10,0)--(10,1)node[above]{$r_{i}$};
    \draw[color=blue] (18,0)--(18,1)node[above]{$r_1$};
    \draw[color=red] (22,0)--(22,2)node[above]{$\hat{r}_1=0$};
    \draw[color=red] (2,0)--(2,2)node[above]{$\hat{r}_{n+1}$};
    \draw[color=red] (8,0)--(8,2)node[above]{$\hat{r}_{i+1}$};
    \draw[color=red] (13,0)--(13,2)node[above]{$\hat{r}_{i}$};
    \draw[color=blue] (2,0)--(2,-1)node[below]{$s_n$};
    \draw[color=blue] (8,0)--(8,-1)node[below]{$s_{i}$};
    \draw[color=blue] (13,0)--(13,-1)node[below]{$s_{i-1}$};
    \draw[color=red] (20,0)--(20,-2)node[below]{$\hat{s}_1$};
    \draw[color=red] (1,0)--(1,-2)node[below]{$\hat{s}_{n}$};
    \draw[color=red] (11,0)--(11,-2)node[below]{$\hat{s}_{i}$};
    \draw (-2,-3) rectangle (24,3);
\end{tikzpicture}
\label{subfig: num nonzero phi 0}
}
\caption{Various cases for the roots/poles of $\Phi_i(p/q)=\hat{p}/\hat{q}$ compared to those of $p/q$. Roots of $p,\hat{p}$ listed as $r_i,\hat{r}_i$ in order above the axis, with roots of $q,\hat{q}$ listed as $s_i,\hat{s}_i$ below the axis. Roots/poles of $p/q$ in blue, with those of $\hat{p}/\hat{q}$ in red with larger vertical offset. When roots of $p$,$\hat{p}$ coincide they are drawn in purple at the same vertical offset.}
\label{fig: all the subfigures}\end{center}
\end{figure}
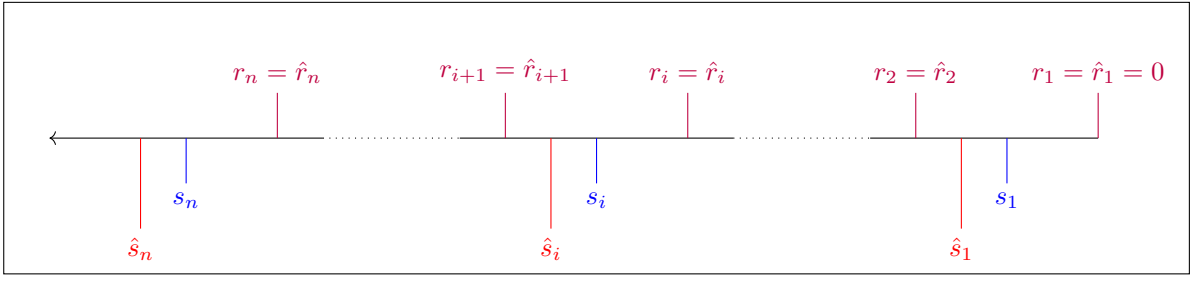
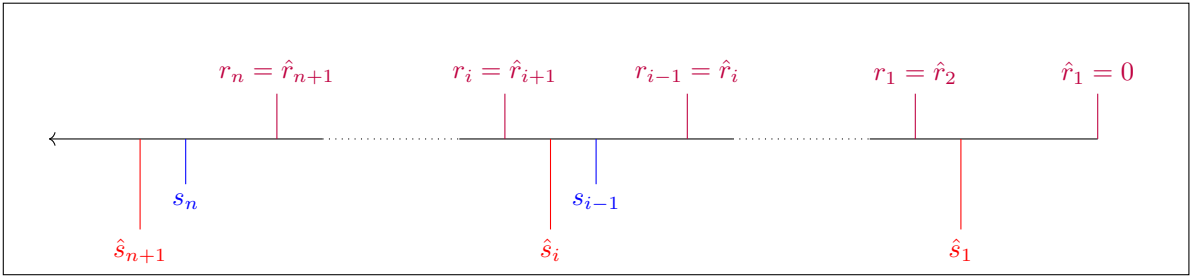
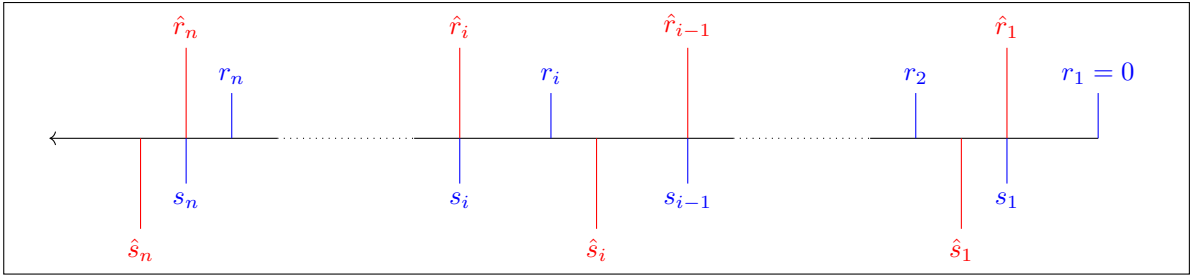
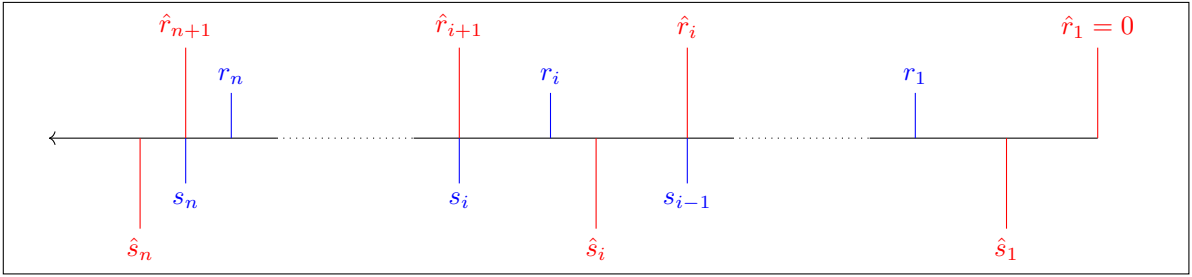

\begin{lemma}\label{lem: subtree roots to -infty}
    For any vertex $v=[a_1,\ldots,a_k]_x=p/q$, let $r_i$ and $s_i$ be the ordered roots and poles for $i=1,\ldots,n$. Then let $\hat{s}_i(N)$ denote the ordered poles of $\Phi_1^{N-1} \circ \Phi_0(v)=[N,a_1,a_2,\ldots,a_k]$. Then
    \[ \lim_{N \rightarrow \infty} \hat{s}'_{i}(N)=\begin{cases}
        s_{i} & (i < n)\\ 
         -\infty & (i=n+1)\\ 
    \end{cases}\]
    \begin{proof}
        There is not any difference if $p(0)=0$ or $p(0) \neq 0$; one may verify that in either case,
        \[ \Phi_1 \circ \Phi_0 \left( \frac{p}{q}\right) = \frac{xq}{(x+1)q+p}\]
        and now we have a numerator that evaluates zero to zero: for any $N \geq 2$, from Lemma  \ref{lem: CF notation for vertices}  we have
        \[ \Phi_1^{N-1} \circ \Phi_0 \left( \frac{p}{q}\right) = \frac{xq}{(x+N-1)q+p}.\]
        The poles are given by solutions to
        \[\frac{p(x)}{q(x)} = -x-(N-1).\]
        On each interval $(s_{i},r_{i}]$, the function $p/q$ is increasing from $-\infty$ to zero, so as $N \rightarrow \infty$ we must approach where $p/q\rightarrow -\infty$, i.e. $s_{i}$.
        To the left of $s_n$ we have a vertical asymptote heading to $+ \infty$, and then as $x \rightarrow -\infty$,  $p/q$ decreases to one (see \Cref{fig: rational plot} for an example). Specifically, $p/q>1$ on $(-\infty,s_n)$. Therefore, if $p/q=-x-(N-1)$ we find that $x<-N$.
    \end{proof}
\end{lemma}
\begin{lemma}\label{lem: subtree roots to zero}
    For any vertex $v=[a_1,\ldots,a_k]_x=p/q$, let $r_i$ and $s_i$ be the ordered roots and poles for $i=1,\ldots,n$. Then let $\hat{s}_i(N)$ denote the ordered poles of $\left(\Phi_1 \circ \Phi_0\right)^{N}(v)=[2,2,\ldots,2,a_1,a_2,\ldots,a_k]$ where there are $N$ consecutive twos. Then
    \[ \lim_{N \rightarrow \infty} \hat{s}_{1}(N)=0\]
    \begin{proof}
        As remarked in \Cref{lem: subtree roots to -infty}, $\Phi_1\circ \Phi_0$ will always create a $p/q$ with $p(0)=0$, so without loss of generality we may assume this holds for the vertex $v$ we begin with. Then
        \[\Phi_1 \circ \Phi_0 \left( \frac{p}{q}\right) = \frac{xq}{(x+1)q+p},\]
        and we are interested in the unique solution to 
        \[ \frac{p(x)}{q(x)} = -(x+1)\]
        in the interval $(s_1,0]$. In this interval, we have $(x-r_i)/(x-s_i)<1$ for $i \geq 2$, so the above becomes
        \[-(x+1) = \frac{(x-r_1)}{(x-s_1)} \cdot \frac{(x-r_2)\cdots (x-r_n)}{(x-s_2)\cdots (x-r_n)}< \frac{x-r_1}{x-s_1}=\frac{x}{x-s_1}\]
        As $x-s_1>0$ we transform this into
        \begin{align*}
        (x+1)(x-s_1)+x&>0    \\
        x^2+(2-s_1)x-s_1&>0
        \end{align*}
        This quadratic has roots at
        \[x_{\pm}=\frac{-(2-s_1)\pm\sqrt{(2-s_1)^2+4s_1}}{2}=\frac{(s_1-2)\pm \sqrt{s_1^2+4}}{2}.\]
        So the unique pole in $(s_1,0]$ must either be larger than $x_+$ or lesser than $x_-$. But
        \[ x_-=\frac{(s_1-2)-\sqrt{s_1^2+4}}{2}<\frac{s_1-2-\max(|s_1|,2)}{2}=\begin{cases}
            \frac{s_1}{2}-2 & s_1\geq -2\\ s_1-1 & s_1 \leq -2
        \end{cases} \qquad \leq s_1 \textrm{ in either case}\]
        So the unique pole in the interval $(s_1,0]$ is in fact at least as large as $x_+$:
        \[x_+= \frac{(s_1-2)+\sqrt{s_1^2+4}}{2}>\frac{s_1-2+2}{2}=\frac{s_1}{2}\]
        Note that we also have:
        \[x_+= \frac{(s_1-2)+\sqrt{s_1^2+4}}{2}<\frac{(s_1-2)+|s_1|+2}{2}=0.\]
        
        In other words, the largest pole of $\Phi_1\circ \Phi_0(v)$ must be at least half as large as the largest pole of $v$. The result follows.        
    \end{proof}
\end{lemma}

We now begin our work towards proving \Cref{th:dense}. 
\begin{lemma}
Suppose that $v_0=[a_1,\ldots,a_k]_x=p/q$ is any vertex with $p(0)=0$ and $q$ is of degree at least two. Let $s_n,\ldots,s_1$ be the ordered roots of $q$, and let $\alpha \in (s_i,s_{i-1})$ for some choice of $i \geq 2$. Let
\[
b=b(v,\alpha)=\max\left\{1,1-\floor{\frac{p_i(\alpha)}{q_i(\alpha)}+\alpha}\right\} = \begin{cases} 1 &  \alpha+ \frac{p(\alpha)}{q(\alpha)} \geq 0 \\ 1-\floor{\frac{p_i(\alpha)}{q_i(\alpha)}+\alpha} &  \alpha + \frac{p(\alpha)}{q(\alpha)}<0.\end{cases}
\]
For each positive integer $n$, let $\hat{s}(n)$ be the pole of $[n,a_1,\ldots,a_k]=x/(x+(n-1)+p/q)$ which is closest to $\alpha$ without being larger than $\alpha$. Then 
\[\hat{s}(b) = \max_{n=1,2,\ldots} \hat{s}(n).\]
In other words, among all $[n,a_1,\ldots,a_k]$, setting $n=b$ has the closest possible pole to $\alpha$ which is not larger than $\alpha$.
\begin{proof}
As in the proof of \Cref{lem: subtree roots to -infty}, the poles of any such $[n,a_1,\ldots,a_k]$ are given by solutions to
\[ x+\frac{p(x)}{q(x)} = -(n-1).\]
For any integer (positive or negative), define $\zeta_n$ to be the unique $x \in (s_i, s_{i-1})$ which solves the above, i.e.
\[n =1 - \left(\zeta_n + \frac{p(\zeta_n)}{q(\zeta_n)}\right)\]
As $x+p/q$ is increasing, these $\zeta$ are decreasing, and with poles $s_i, s_{i-1}$ we have both $\zeta_n \rightarrow s_i$ as $n \rightarrow \infty$, and $\zeta_n \rightarrow s_{i-1}$ as $n\rightarrow -\infty$.
So we select $b$ to be the smallest positive integer so that $\zeta_n \leq \alpha$.

If $\alpha+p(\alpha)/q(\alpha)<0$, then we have
\[ \zeta_b \leq \alpha < \zeta_{b-1}.\]
As $p/q$ is increasing, we similarly obtain
\[ \frac{p(\zeta_b)}{q(\zeta_b)}+\zeta_b \leq \frac{p(\alpha)}{q(\alpha)}+\alpha<\frac{p(\zeta_{b-1})}{q(\zeta_{b-1})}+\zeta_{b-1}.\]
From our definition of $\zeta_n$, then:
\[-(b-1) \leq \frac{p(\alpha)}{q(\alpha)}+\alpha<-(b-2),\]
from which we directly conclude that
\[ b =1- \floor{\frac{p(\alpha)}{q(\alpha)}+\alpha},\]
and under the assumption that $\alpha+p(\alpha)/q(\alpha)<0$ this is certainly not less than one.

If $\alpha+p(\alpha)/q(\alpha)\geq 0$, then certainly $b=1$ is the smallest positive integer $n$ so that $\zeta_n \leq \alpha$, and the claim directly follows.
\end{proof}
\end{lemma}

We may now inductively define a sequence of vertices, a path in our tree beginning at some vertex $v$. We will show that if $\alpha<-1$, this path generates a sequence of poles converging to $\alpha$.

For any $v_0=[a_1,\ldots,a_k]_x=p_0/q_0$ a vertex with $p(0)=0$ and $q$ of degree at least two, and $\alpha$ between the least and largest poles of $v$ while not being equal to a pole of any vertex in the subtree rooted at $v_0$, inductively define $b_{i+1} = b(v_i,\alpha)$ where
\begin{equation}
    \label{eqn: function b}
    b(p/q,\alpha) = \max\{1,1-\floor{\frac{p(\alpha)}{q(\alpha)}+\alpha}\} = \begin{cases} 1 &  \alpha+ \frac{p(\alpha)}{q(\alpha)} \geq 0 \\ 1-\floor{\frac{p(\alpha)}{q(\alpha)}+\alpha} &  \alpha + \frac{p(\alpha)}{q(\alpha)}<0.\end{cases}
\end{equation}

These $b_i$ are a function of both $\alpha$ and $v_0$, but as both will generally be fixed we frequently suppress that notation.

For any $v_0=[a_1,\ldots,a_k]_x=p/q$ and $\alpha <0$, compute $b_i=b(v_i,\alpha)$ as in \Cref{eqn: function b} and define 
\[v_{i} = [b_i,v_{i-1}] = \frac{x}{x+(b_i-1) + \frac{p_{i-1}(x)}{q_{i-1}(x)}} = \frac{x}{x-\floor{\alpha+ \frac{p_{i-1}(\alpha)}{q_{i-1}(\alpha)}} + \frac{p_{i-1}(x)}{q_{i-1}(x)}}.\]

\begin{lemma}
    \label{lem: b is one only once}
    For any $v_0=[a_1,\ldots,a_k]_x=p_0/q_0$, for any $\alpha<0$, every $b_i \geq 2$ except perhaps $b_1=1$.
    \begin{proof}
        From \Cref{eqn: function b} we see that if $\alpha+p_i(\alpha)/q_i(\alpha)<0$, then we have $b_{i+1} \geq 2$. But for any $v_i$, let $s<\alpha<\hat{s}$ be two consecutive roots of $q_i$. Then in $v_{i+1}$ both $s$ and $\hat{s}$ are two consecutive \textit{roots}:
        \[v_{i+1} = \frac{x}{x+b_{i+1}-1+\frac{p_i}{q_i}}=\frac{xq_i}{(x+b_{i+1}-1)q_i+p_i}.\]
        As our choice of $b_{i+1}$ corresponded to $\zeta_{b_{i+1}}$ being a pole of $v_{i+1}$, and $\zeta_{b_{i+1}}$ is \textit{less} than $\alpha$, we have $\alpha$ between a pole and a root of the increasing function $v_{i+1}$, so $p_{i+1}(\alpha)/q_{i+1}(\alpha)<0$. Therefore $b_{i+2} \geq 2$. So the only $b_i$ which could be less than two is $b_1$ (and then if and only if $\alpha+p_0(\alpha)/q_0(\alpha)>0$).
    \end{proof}
\end{lemma}

\begin{proof}[Proof of \Cref{th:dense}]
    If any $v_i$ has a pole exactly equal to $\alpha$, then certainly $\alpha$ is in the closure of $R$ (the set of all poles of vertices in the subtree rooted at $v_0$).

    Otherwise, for each $i=0,1,\ldots$ let $\zeta_i,\hat{\zeta}_i$ be the solutions to 
    \[\zeta_i+\frac{p_i(\zeta_i)}{q_i(\zeta_i)} = b_i-1, \qquad\hat{\zeta}_i+\frac{p_i(\hat{\zeta}_i) }{q_i(\hat{\zeta}_i) }= b_i\]
    These satisfy several useful relations, all of which follow from how we defined $b_i$:
    \[ \zeta_1<\zeta_2<\cdots<\alpha<\hat{\zeta}_i\] 
    We do not claim any obvious ordering on the $\hat{\zeta}_i$, but they are all certainly larger than $\alpha$. If the $\zeta_i$ converge to $\alpha$, then the result follows as well. So assume to the contrary that the $\zeta_i$ converge to some $S<\alpha$.

    The function $x+v_i(x)$ is strictly increasing from $1-b_i$ to $2-b_i$ on the interval $[\zeta_i,\hat{\zeta}_i]$. It follows now that on the interval $[S,\alpha]$ each $x+v_i(x)$ increases by less than one, and so the average value of $1+v'_i(x)$ must be less than $1/(\alpha-S)$.

    But if we compute (on the interval $[S,\alpha]$):
    \begin{align*}
        \left(x + v_{i+1}(x)\right)' &= 1 + \left( \frac{x}{x+(b_{i+1}-1)+v_i(x)}\right)'\\
        &=1 + \frac{(b_{i+1}-1)+v_i(x)-xv_i'(x))}{\left(x+(b_{i+1}-1)+v_i(x)\right)^2}\\
        &=1+  \frac{1}{x+(b_{i+1}-1)+v_i(x)}+\frac{-x(v'_i(x)+1)}{\left(x+(b_{i+1}-1)+v_i(x)\right)^2}.\\
        \intertext{Since $x<0$ and $v'_i>0$ we replace the increasing denominator with its value at $\alpha$:}
        &>\frac{1}{\alpha+b_{i+1}-1+v_i(\alpha)} + \frac{-x(v'_i+1)}{\left(\alpha+b_{i+1}-1+v_i(\alpha)\right)^2}.\\
        \intertext{From \Cref{eqn: function b} we now have:}
        &>\frac{1}{\left\{\alpha+v_i(\alpha)\right\}} + \frac{-xv'_i(x)}{\left\{\alpha+v_i(\alpha)\right\}^2}\\
        &> (-\alpha) \cdot \inf_{x \in [S,\alpha]}(v'_i(x)).
    \end{align*}
    It now follows that if we denote $C \geq 1$ as the infimum of $(x+v_0(x))'$ on $[S,\alpha]$, that for $\alpha<-1$ we have for all $i$ and all $x \in [S,\alpha]$
    \[(x+v_i(x))' > C(-\alpha)^i.\]
    We have already remarked that the average value of $1+v'_i(x)$ on $[S,\alpha]$ must be no larger than $1/(\alpha-S)$ for all $i$, contradicting the above and completing the proof.
\end{proof}

\begin{remark}
    Computational evidence suggests that \Cref{th:dense} may be extended to $(-\infty,0]$, which is somewhat supported by \Cref{lem: subtree roots to zero}. The same proof as presented above applies if one may prove that $\{\alpha+v_i(\alpha)\}$ is not bounded away from zero, at least for a dense set of $\alpha \in (-1,0)$.
\end{remark}

\section{Backward Farey Tree}\label{sec:other trees}
The results we have on the Farey polynomial tree raise the question of how general these results are: there are many other different families of continued fraction algorithms. In this section we take the backward continued fractions as a starting point and compare it with Theorem \ref{th:cfs} and Theorem \ref{th:dense}. 
Just as for regular continued fractions, the backward continued fractions appears in different fields of mathematics. There are many nice papers studying backwards continued fractions, but we want to highlight \cite{MO19,MO20} where backwards (and regular) continued fractions are used to draw connections between graph theory, combinatorics, number theory, and algebra. 
For backward continued fractions, a parametrised family is already studied in the form of $ 1/(u(1-x)) \mod 1$ where $u\in(0,4)$, introduced and studied in \cite{GH96a,GH96b} and more recently \cite{LS20, LS22}. In contrast to those papers, in this article we will make the branches full again by shifting the digit set in order to get a better comparison with the  Farey polynomial tree. Of course, instead of having  $ 1/(u(1-x)) \mod 1$, we can also put the parameter in the numerator. This leads to the fast map 
\begin{equation}
    T_x(t)=\frac{x}{1-t}- x -\left\lfloor \frac{x}{1-t}-x  \right\rfloor.
\end{equation}
The continued fractions are of the form 
\[
t=1-\cfrac{x}{x + a_1-\cfrac{x}{x+ a_2 -\ddots}}
\]
where $a_i\in\mathbb{Z^+}$, given by
\[a(t)=1+\floor{\frac{x}{1-t}-x}.\] Note that, when taking $x\in \mathbb{Z^+}$, you essentially get $N$-continued fractions with a negative value for $N$. 
The slow map is given by
\begin{equation}\label{eq:F_x backwards}
  F_x(t)=
\begin{cases}
\dfrac{xt}{1-t}   & t\in [0,\frac{1}{x+1})\\[1ex]
1-\displaystyle\frac{x(t-1)}{1-t-x}, & t\in [\frac{1}{x+1},1],
\end{cases}  
\end{equation}
 see Figure \ref{fig:Fareymapbackward}.

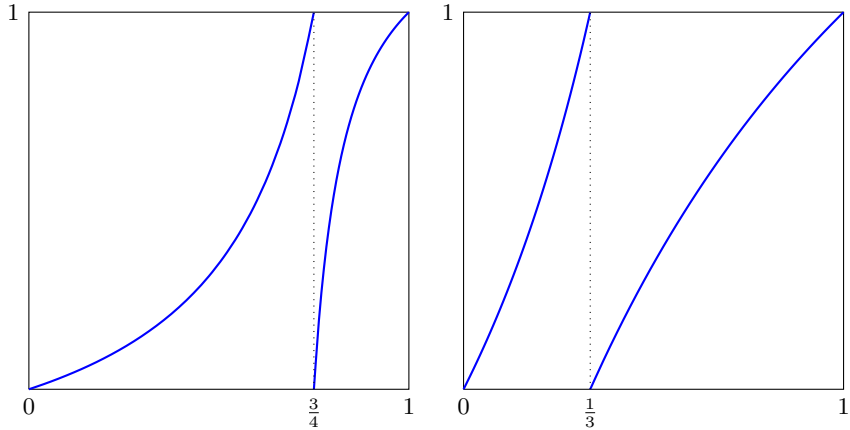
\begin{figure}[ht]
		\centering
		
		\subfigure{\begin{tikzpicture}[scale=5]
				\draw[white] (-0.25,0)--(1,0);
				\draw(0,0)node[below]{\small $0$}--(1,0)node[below]{\small $1$}--(1,1)--(0,1)node[left]{\small $1$}--(0,0);

				\draw[thick, blue, smooth, samples =20, domain=3/4:1] plot(\x,{1-(\x-1)/(2-3*\x)});
				\draw[thick,blue, smooth, samples =20, domain=0:3/4] plot(\x,{\x/(3-3*\x)});

				\draw[dotted](3/4,0)node[below]{\small $\frac{3}{4}$}--(3/4,1);
				
		\end{tikzpicture}}
        \subfigure{\begin{tikzpicture}[scale=5]
				\draw(0,0)node[below]{\small $0$}--(1,0)node[below]{\small $1$}--(1,1)--(0,1)node[left]{\small $1$}--(0,0);

				\draw[thick, blue, smooth, samples =20, domain=1/3:1] plot(\x,{1-(2*\x-2)/(-\x-1)});
				\draw[thick,blue, smooth, samples =20, domain=0:1/3] plot(\x,{(2*\x)/(1-\x)});

				\draw[dotted](1/3,0)node[below]{\small $\frac{1}{3}$}--(1/3,1);
				
		\end{tikzpicture}}
		
		\caption{The map $F_x$ for $x=1/3$ on the left and $x=2$ on the right. \hspace{5.5em} }
		\label{fig:Fareymapbackward}
	\end{figure}

For the corresponding tree, which we will call the backward tree, we use the discontinuity as the root. Here, $\Phi_1$ and $\Phi_0$ are given by the inverse branches of $F_x$ which results in

\begin{equation}\label{eqn: Phibackward}\begin{split}
\Phi_0\left( \frac{p(x)}{q(x)}\right) &= \frac{(x-1)p(x)+q(x)}{(x+1)q(x)-p(x)} \\
\\
\Phi_1\left( \frac{p(x)}{q(x)} \right) &=\frac{p(x)}{xq(x)+p(x)}, \end{split}\end{equation}
 see Figure \ref{fig:FareyPolybackward} for the first 4 levels. Note that we need to induce on the left interval (i.e. $[0,\frac{1}{x+1})$) to get the fast map. Concerning the analogy with Theorem \ref{th:cfs} we find a relatively similar result.

\begin{figure}[h]
    \begin{center}

\begin{tikzpicture}[
    level 1/.style = {sibling distance=7.5cm},
    level 2/.style = {sibling distance=4cm},
    level 3/.style = {sibling distance=2cm},
    level distance          = 1.5cm,
    edge from parent/.style = {draw},
    scale=1.2
    ]

    \node {$\frac{1}{x+1}$}
    child{
      node {$\frac{1}{x^2+x+1}$}
      child{
        node {$\frac{1}{x^3+x^2+x+1}$}
        child{
          node {$\frac{1}{x^4+x^3+x^2+x+1}$}
        }
        child{
          node {$\frac{x^2+x+2}{x^3+2x^2+2x+2}$}
        }
      }
      child{
        node {$\frac{x+2}{x^2+2x+2}$}
        child{
          node {$\frac{x+2}{x^3+2x^2+3x+2}$}
        }
        child{
          node {$\frac{2x+3}{x^2+3x+3}$}
        }
      }
    }
    child{
      node {$\frac{2}{x+2}$}
      child{
        node {$\frac{2}{x^2+2x+2}$}
        child{
          node {$\frac{2}{x^3+2x^2+2x+2}$}
        }
        child{
          node {$\frac{x+4}{x^2+3x+4}$}
        }
      }
      child{
        node {$\frac{3}{x+3}$}
        child{
          node {$\frac{3}{x^2+3x+3}$}
        }
        child{
          node {$\frac{4}{x+4}$}
        }
      }
    };
\end{tikzpicture}
     
        \caption{The first four levels general backward Farey tree. The fractions are simplified.}\label{fig:FareyPolybackward}
    \end{center}
\end{figure}
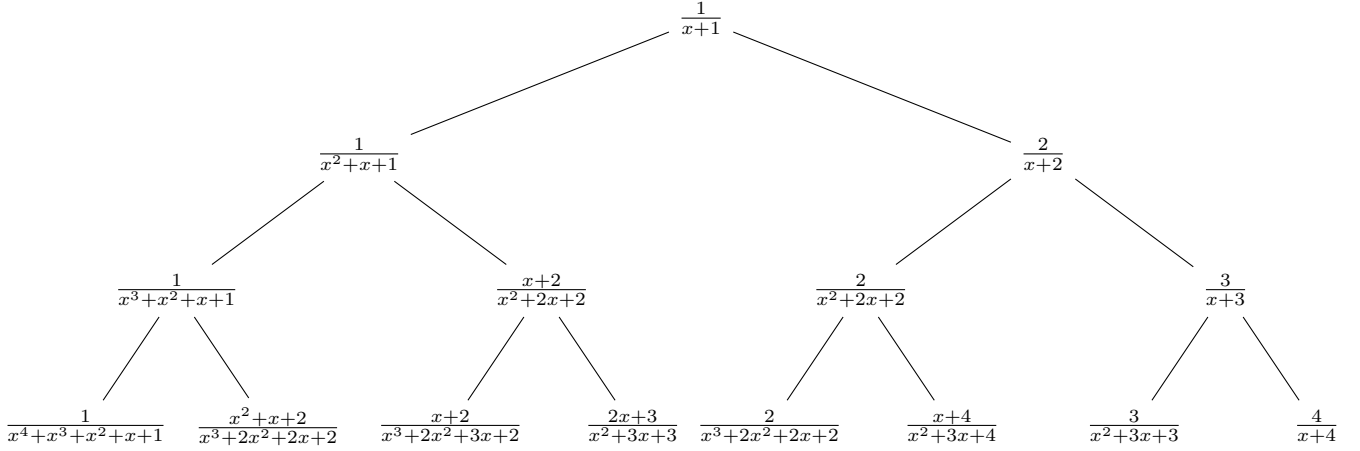

\begin{theorem}\label{th:backwardcf}
    For a fixed value of $x$, we consider the function which maps the backward tree to the set of values given by evaluating every vertex at this value of $x$.
    \begin{enumerate}
        \item Setting a value $x\geq 1$ injectively maps the backward tree to a dense subset of $(0,1)$. \\For $x\in(0,1)$, the mapping is injective and maps to a subset of $(0,1)$ but the image is not dense.
        \item Setting a value $x \geq 1$ bijectively maps the tree to $\mathbb{Q} \cap (0,1)$ if and only if $x \in \mathbb{Z}^+$.
    \end{enumerate}
\end{theorem}

To  prove this we will use the following two propositions.

\begin{proposition}\label{prop:convergenceback}
    For $x>0$, we have convergence of $p_n(t)/q_n(t)\rightarrow t$ for all $t\in(0,1)$ if and only if  $x \geq 1$.
\end{proposition}
\begin{proof}
  
In analogy with \cite{LS22}, we get the formulas 

\begin{equation}\label{eq:rec for pn and qn backward}
    \begin{aligned}
        p_{0}=1,&\, p_1=a_1, & p_n=(x+a_n)p_{n-1}-xp_{n-2},~\text{for } n\geq 1,\\
    q_{0}=1, &\, q_1=a_1+x, & q_n=(x+a_n)q_{n-1}-xq_{n-2},~\text{for } n\geq 1.
    \end{aligned}
\end{equation}
and
\begin{equation}
    \left| t-\frac{p_n}{q_n} \right|\leq\frac{x^n}{q_n(q_n-q_{n-1})}.
\end{equation}
Just as in the proof of \Cref{prop: convergence for 0<x<1}, we achieve the worst convergence when all $a_n=1$, which gives $q_n=\sum_{j=0}^n x^j$ and $q_n-q_{n-1}=x^n$. Substituting this in the right hand side gives 
\[
\frac{x^n}{q_n(q_n-q_{n-1})}=\frac{x^n}{x^n\sum_{j=0}^n x^j}=\frac{1}{\sum_{j=0}^n x^j}
\]
which goes to zero when $n$ goes to infinity if $x \geq 1$. Note that for $x\in(0,1)$, the map $F_x(t)$ has an attracting fixed point. Indeed, $F_x^\prime(0)=x$, which results in $F_x^n(t)\rightarrow 0$ (but never reaching zero) as $n\rightarrow \infty$ for all $t\in(0,1-x)$ which prevents convergence.
\end{proof}

We would like to remark that the range of parameters for which we find convergence is different from the previously referenced papers where the branches were not full. In those references, the convergence is guaranteed for $u\in(0,4)$ which translates to $x\in(\frac{1}{4},\infty)$ when putting the parameter in the numerator. This is because in the non-full branch case, the sequences of digits with the worst convergence are not admissible.
\begin{proposition}\label{prop:finiteback}
    For $x\in\mathbb{Q}_{>0}$, all rationals in $(0,1)$ have a finite backward $x$-expansion if and only if $x\in\mathbb{N}$.  
\end{proposition}
\begin{proof}
Let $\frac{t_0}{s_0}\in(0,1)$ coprime and set $T_x^n(\frac{t_0}{s_0})=\frac{t_n}{s_n}$ for any $n$ it makes sense.\\ For $x\in \mathbb{N}$ we have the following.
\[
T_x\left(\frac{t_n}{s_n}\right)=\frac{t_n x-(a-1)(s_n-t_n)}{s_n-t_n}=\frac{t_{n+1}}{s_{n+1}}.
\]
This gives us $s_{n+1}=s_n-t_n<s_n$ even without dividing out common divisors. Note that here we took the denominators instead of the numerators as for the forward case (and in fact the sequence of numerators does not need to be strictly decreasing). 
Therefore the denominators decrease when applying $T_x$ and will be $0$ at some point.\\
For $x\in\mathbb{Q}\backslash \mathbb{N}$ we will do the same as the forward case. The only difference is that the computation will be slightly different. We write $x=p/q$ with $p,q$ coprime and assume that $t_n,s_n$ are coprime.
\[
T_x\left(\frac{t_n}{s_n}\right)=\frac{t_n p-(a-1)(s_n-t_n) q}{(s_n-t_n)q}=\frac{t_{n+1}}{s_{n+1}}.
\]
Now for $s_n\equiv 0 \mod q$ we find that $t_{n+1} \mod q \equiv  t_np \mod q \not\equiv 0 \mod q $ and $s_{n+1}\equiv 0 \mod q$. In other words, if we start with $s_0$ divisible by $q$ and $t_0$ not divisible by $q$ the numerators of any point in the orbit will never be $0 \mod q$ and therefore never can be $0$. (On the other hand the denominators will always be divisible by $q$.)
\end{proof}

\begin{proof}[Proof of Theorem \ref{th:backwardcf}]
Point (1) immediately follows from the fact that the inverse branches of $F_x$ map to disjoint sets, and then Proposition \ref{prop:convergenceback} gives denseness. Point (2) follows immediately from Proposition~\ref{prop:finiteback}. 
\end{proof}

When we look at the roots of $p(x)$ and $q(x)$ of the vertices in the backward tree instead of the forward tree, the story is very different. From \Cref{eqn: Phi} we see that the roots of $p(x)$ and $q(x)$ appearing in the forward tree are the same set with the unique exception that $x=0$ is never a root of $q(x)$. For the backward tree we cannot immediately conclude that the roots of $p,q$ have any obvious relationship. Also, for the forward Farey tree we found only real roots that are dense in $(-\infty,-1]$. For the backward Farey tree we find infinitely many complex roots, both for $p(x)$ as for $q(x)$, see Figure \ref{fig:rootsbackward}. A simple proof shows the following:

\begin{proposition}
    The roots of $q(x)$ are dense on the unit circle.
\end{proposition}
\begin{proof}
    Note that $\Phi_1^n(1/(x+1))=1/(\sum_{j=0}^{n+1}x^j)$. Now $(x-1)\sum_{j=0}^{n+1}x^j=x^{n+2}-1$ so that the roots of $\sum_{j=0}^{n+1}x^j$ are $\{x= e^{(k\pi i)/n} : k\in \{1,\ldots ,n+1\} \}$. Of course, we then get a dense subset of the unit circle if we take the union over all $n\in \mathbb{Z}^+$.
\end{proof}
\begin{figure}
    \centering
    \includegraphics[width=0.45\linewidth]{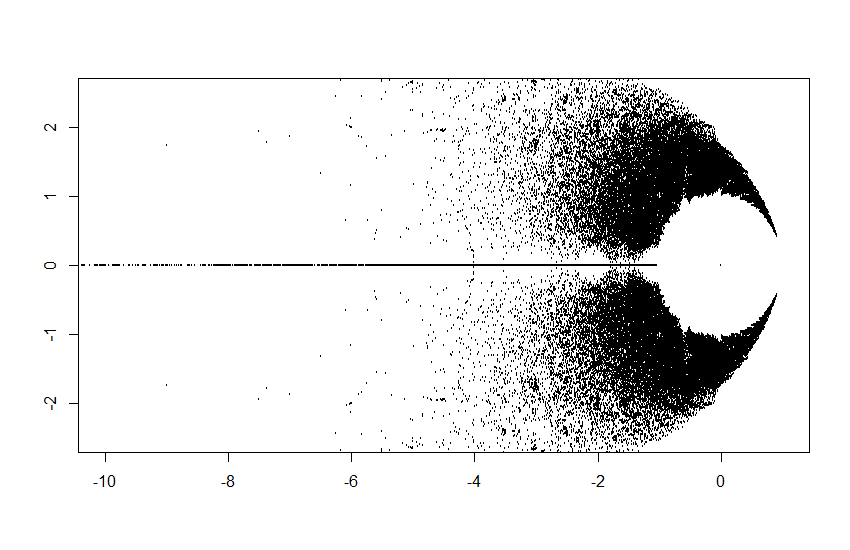}
    \includegraphics[width=0.45\linewidth]{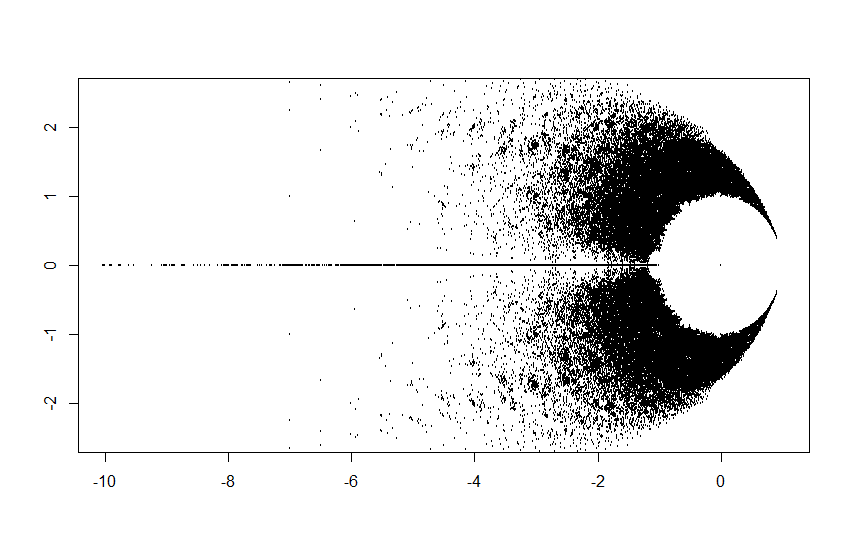}
    \caption{The roots of the vertices in the complex plane for $p(x)$ on the left and for $q(x)$ on the right. The first 15 levels of the tree are used.}
    \label{fig:rootsbackward}
\end{figure}

\bibliographystyle{alpha}
\bibliography{biblio}

@article {AK22,
    AUTHOR = {Aiylam, Dhroova and Khovanova, Tanya},
     TITLE = {Stern-{B}rocot sequences from weighted mediants},
   JOURNAL = {J. Number Theory},
  FJOURNAL = {Journal of Number Theory},
    VOLUME = {238},
      YEAR = {2022},
     PAGES = {313--330},
      ISSN = {0022-314X,1096-1658},
   MRCLASS = {11B57},
  MRNUMBER = {4430102},
MRREVIEWER = {Jeffrey\ O.\ Shallit},
       DOI = {10.1016/j.jnt.2021.08.014},
       URL = {https://doi.org/10.1016/j.jnt.2021.08.014},
}

@article {AW11,
	AUTHOR = {Anselm, Maxwell and Weintraub, Steven H.},
	TITLE = {A generalization of continued fractions},
	JOURNAL = {J. Number Theory},
	FJOURNAL = {Journal of Number Theory},
	VOLUME = {131},
	YEAR = {2011},
	NUMBER = {12},
	PAGES = {2442--2460},
	ISSN = {0022-314X},
	CODEN = {JNUTA9},
	MRCLASS = {11A55},
	MRNUMBER = {2832836},
	MRREVIEWER = {James G. Mc Laughlin},
	DOI = {10.1016/j.jnt.2011.06.007},
	URL = {http://dx.doi.org/10.1016/j.jnt.2011.06.007},
}

@article {BBDG24,
    AUTHOR = {Baalbaki, Wael and Bonanno, Claudio and Del Vigna, Alessio and
              Garrity, Thomas and Isola, Stefano},
     TITLE = {On integer partitions and continued fraction type algorithms},
   JOURNAL = {Ramanujan J.},
  FJOURNAL = {Ramanujan Journal. An International Journal Devoted to the
              Areas of Mathematics Influenced by Ramanujan},
    VOLUME = {63},
      YEAR = {2024},
    NUMBER = {3},
     PAGES = {873--915},
      ISSN = {1382-4090,1572-9303},
   MRCLASS = {11P81 (11B57 11J70 37A44)},
  MRNUMBER = {4707386},
MRREVIEWER = {N.\ G.\ Moshchevitin},
       DOI = {10.1007/s11139-023-00791-5},
       URL = {https://doi.org/10.1007/s11139-023-00791-5},
}

@article {BI09,
    AUTHOR = {Bonanno, Claudio and Isola, Stefano},
     TITLE = {Orderings of the rationals and dynamical systems},
   JOURNAL = {Colloq. Math.},
  FJOURNAL = {Colloquium Mathematicum},
    VOLUME = {116},
      YEAR = {2009},
    NUMBER = {2},
     PAGES = {165--189},
      ISSN = {0010-1354,1730-6302},
   MRCLASS = {37A45 (11A55 11B57 37E25)},
  MRNUMBER = {2520138},
MRREVIEWER = {Tobias\ M\"uhlenbruch},
       DOI = {10.4064/cm116-2-3},
       URL = {https://doi.org/10.4064/cm116-2-3},
}

@article {BGRKWY208,
	AUTHOR = {Burger, Edward B. and Gell-Redman, Jesse and Kravitz, Ross and
	Walton, Daniel and Yates, Nicholas},
	TITLE = {Shrinking the period lengths of continued fractions while
	still capturing convergents},
	JOURNAL = {J. Number Theory},
	FJOURNAL = {Journal of Number Theory},
	VOLUME = {128},
	YEAR = {2008},
	NUMBER = {1},
	PAGES = {144--153},
	ISSN = {0022-314X},
	CODEN = {JNUTA9},
	MRCLASS = {11J70 (11K60)},
	MRNUMBER = {2382774},
	MRREVIEWER = {James G. Mc Laughlin},
	DOI = {10.1016/j.jnt.2007.03.001},
	URL = {http://dx.doi.org/10.1016/j.jnt.2007.03.001},
}

@article {BDM21,
    AUTHOR = {Bonanno, Claudio and Del Vigna, Alessio and Munday, Sara},
     TITLE = {A slow triangle map with a segment of indifferent fixed points
              and a complete tree of rational pairs},
   JOURNAL = {Monatsh. Math.},
  FJOURNAL = {Monatshefte f\"ur Mathematik},
    VOLUME = {194},
      YEAR = {2021},
    NUMBER = {1},
     PAGES = {1--40},
      ISSN = {0026-9255,1436-5081},
   MRCLASS = {37A40 (11K50 37A44)},
  MRNUMBER = {4200969},
MRREVIEWER = {Thomas\ Ward},
       DOI = {10.1007/s00605-020-01500-w},
       URL = {https://doi.org/10.1007/s00605-020-01500-w},
}

@article{CK,
 author = {Chen, Yufei and Kraaikamp, Cor},
 title = {Matching of orbits of certain {{\(N\)}}-expansions with a finite set of digits},
 fjournal = {T{\^o}hoku Mathematical Journal. Second Series},
 journal = {T{\^o}hoku Math. J. (2)},
 issn = {0040-8735},
 volume = {77},
 number = {3},
 pages = {319--343},
 year = {2025},
 language = {English},
 doi = {10.2748/tmj.20230802},
 url = {projecteuclid.org/journals/tohoku-mathematical-journal/volume-77/issue-3/Matching-of-orbits-of-certain-N-expansions-with-a-finite/10.2748/tmj.20230802.full},
 zbMATH = {8091382}
}

@article{DKS25,
 author = {Dajani, Karma and Kraaikamp, Cor and Sanderson, Slade},
 title = {A unifying theory for metrical results on regular continued fraction convergents and mediants},
 fjournal = {Mathematics of Computation},
 journal = {Math. Comput.},
 issn = {0025-5718},
 volume = {94},
 number = {356},
 pages = {3101--3144},
 year = {2025},
 language = {English},
 doi = {10.1090/mcom/4046},
 zbMATH = {8083072}
}

@article {DKW13,
    AUTHOR = {Dajani, Karma and Kraaikamp, Cor and van der Wekken, Niels},
     TITLE = {Ergodicity of {$N$}-continued fraction expansions},
   JOURNAL = {J. Number Theory},
  FJOURNAL = {Journal of Number Theory},
    VOLUME = {133},
      YEAR = {2013},
    NUMBER = {9},
     PAGES = {3183--3204},
      ISSN = {0022-314X,1096-1658},
   MRCLASS = {11J70 (11K60 28D05)},
  MRNUMBER = {3057071},
MRREVIEWER = {Andrew\ M.\ Rockett},
       DOI = {10.1016/j.jnt.2013.02.017},
       URL = {https://doi.org/10.1016/j.jnt.2013.02.017},
}

@article {DO18,
    AUTHOR = {Dajani , K. and Oomen, M.},
     TITLE = {Random {$N$}-continued fraction expansions},
   JOURNAL = {J. Approx. Theory},
  FJOURNAL = {Journal of Approximation Theory},
    VOLUME = {227},
      YEAR = {2018},
     PAGES = {1--26},
      ISSN = {0021-9045},
   MRCLASS = {28D05 (11K50 37C40)},
  MRNUMBER = {3763845},
MRREVIEWER = {Thomas Ward},
       DOI = {10.1016/j.jat.2017.11.003},
       URL = {https://doi.org/10.1016/j.jat.2017.11.003},
}

@article {DS07,
    AUTHOR = {Dilcher, Karl and Stolarsky, Kenneth B.},
     TITLE = {A polynomial analogue to the {S}tern sequence},
   JOURNAL = {Int. J. Number Theory},
  FJOURNAL = {International Journal of Number Theory},
    VOLUME = {3},
      YEAR = {2007},
    NUMBER = {1},
     PAGES = {85--103},
      ISSN = {1793-0421,1793-7310},
   MRCLASS = {11B83 (11B37 11B75)},
  MRNUMBER = {2310494},
MRREVIEWER = {Giedrius\ Alkauskas},
       DOI = {10.1142/S179304210700081X},
       URL = {https://doi.org/10.1142/S179304210700081X},
}

@article {GH96a,
    AUTHOR = {Gr\"ochenig, Karlheinz and Haas, Andrew},
     TITLE = {Backward continued fractions and their invariant measures},
   JOURNAL = {Canad. Math. Bull.},
  FJOURNAL = {Canadian Mathematical Bulletin. Bulletin Canadien de
              Math\'ematiques},
    VOLUME = {39},
      YEAR = {1996},
    NUMBER = {2},
     PAGES = {186--198},
      ISSN = {0008-4395,1496-4287},
   MRCLASS = {11J70 (58F11)},
  MRNUMBER = {1390354},
MRREVIEWER = {Richard\ T.\ Bumby},
       DOI = {10.4153/CMB-1996-023-8},
       URL = {https://doi.org/10.4153/CMB-1996-023-8},
}

@article {GH96b,
    AUTHOR = {Gr\"ochenig, Karlheinz and Haas, Andrew},
     TITLE = {Backward continued fractions, {H}ecke groups and invariant
              measures for transformations of the interval},
   JOURNAL = {Ergodic Theory Dynam. Systems},
  FJOURNAL = {Ergodic Theory and Dynamical Systems},
    VOLUME = {16},
      YEAR = {1996},
    NUMBER = {6},
     PAGES = {1241--1274},
      ISSN = {0143-3857,1469-4417},
   MRCLASS = {58F11 (11F06 11J70)},
  MRNUMBER = {1424398},
MRREVIEWER = {Edoh\ Amiran},
       DOI = {10.1017/S0143385700010014},
       URL = {https://doi.org/10.1017/S0143385700010014},
}

@article {GS17,
    AUTHOR = {Greene, John and Schmieg, Jesse},
     TITLE = {Continued fractions with non-integer numerators},
   JOURNAL = {J. Integer Seq.},
  FJOURNAL = {Journal of Integer Sequences},
    VOLUME = {20},
      YEAR = {2017},
    NUMBER = {1},
     PAGES = {Article 17.1.2, 26},
      ISSN = {1530-7638},
   MRCLASS = {11A55},
  MRNUMBER = {3606972},
MRREVIEWER = {Tuangrat\ Chaichana},
}

@book {HW79,
    AUTHOR = {Hardy, G. H. and Wright, E. M.},
     TITLE = {An introduction to the theory of numbers},
   EDITION = {Fifth},
 PUBLISHER = {The Clarendon Press, Oxford University Press, New York},
      YEAR = {1979},
     PAGES = {xvi+426},
      ISBN = {0-19-853170-2},
   MRCLASS = {10-01},
  MRNUMBER = {568909},
MRREVIEWER = {T.\ M.\ Apostol},
}

@article{JKN,
 author = {de Jonge, Jaap and Kraaikamp, Cor and Nakada, Hitoshi},
 title = {Orbits of {{\(N\)}}-expansions with a finite set of digits},
 fjournal = {Monatshefte f{\"u}r Mathematik},
 journal = {Monatsh. Math.},
 issn = {0026-9255},
 volume = {198},
 number = {1},
 pages = {79--119},
 year = {2022},
 language = {English},
 doi = {10.1007/s00605-021-01658-x},
 zbMATH = {7534063},
 Zbl = {1494.11059}
}

@article {KO86,
    AUTHOR = {Kim, Seung-hwan and \"Ostlund, Stellan},
     TITLE = {Simultaneous rational approximations in the study of dynamical
              systems},
   JOURNAL = {Phys. Rev. A (3)},
  FJOURNAL = {Physical Review. A. Third Series},
    VOLUME = {34},
      YEAR = {1986},
    NUMBER = {4},
     PAGES = {3426--3434},
      ISSN = {1050-2947,1094-1622},
   MRCLASS = {58F22 (11K55 11K60)},
  MRNUMBER = {863825},
MRREVIEWER = {Paul\ R.\ Stein},
       DOI = {10.1103/PhysRevA.34.3426},
       URL = {https://doi.org/10.1103/PhysRevA.34.3426},
}

@article{KL,
 author = {Kraaikamp, Cor and Langeveld, Niels},
 title = {On matching and periodicity for {{\((N,\alpha)\)}}-expansions},
 fjournal = {The Ramanujan Journal},
 journal = {Ramanujan J.},
 issn = {1382-4090},
 volume = {64},
 number = {4},
 pages = {1479--1496},
 year = {2024},
 language = {English},
 doi = {10.1007/s11139-024-00878-7},
 zbMATH = {7904957},
 Zbl = {1553.37058}
}

@article {L40,
    AUTHOR = {Leighton, Walter},
     TITLE = {Proper continued fractions},
   JOURNAL = {Amer. Math. Monthly},
  FJOURNAL = {American Mathematical Monthly},
    VOLUME = {47},
      YEAR = {1940},
     PAGES = {274--280},
      ISSN = {0002-9890,1930-0972},
   MRCLASS = {27.0X},
  MRNUMBER = {2567},
MRREVIEWER = {D.\ H.\ Lehmer},
       DOI = {10.2307/2302686},
       URL = {https://doi.org/10.2307/2302686},
}

@article {LRS17,
    AUTHOR = {Leroy, Julien and Rigo, Michel and Stipulanti, Manon},
     TITLE = {Counting the number of non-zero coefficients in rows of
              generalized {P}ascal triangles},
   JOURNAL = {Discrete Math.},
  FJOURNAL = {Discrete Mathematics},
    VOLUME = {340},
      YEAR = {2017},
    NUMBER = {5},
     PAGES = {862--881},
      ISSN = {0012-365X,1872-681X},
   MRCLASS = {05A15 (05A05 11B65 68R15)},
  MRNUMBER = {3612418},
MRREVIEWER = {Zhiyong\ Zhu},
       DOI = {10.1016/j.disc.2017.01.003},
       URL = {https://doi.org/10.1016/j.disc.2017.01.003},
}

@article {LS20,
    AUTHOR = {Sebe, Gabriela Ileana and Lascu, Dan},
     TITLE = {Convergence rate for {R}\'enyi-type continued fraction
              expansions},
   JOURNAL = {Period. Math. Hungar.},
  FJOURNAL = {Periodica Mathematica Hungarica. Journal of the J\'anos Bolyai
              Mathematical Society},
    VOLUME = {81},
      YEAR = {2020},
    NUMBER = {2},
     PAGES = {239--249},
      ISSN = {0031-5303,1588-2829},
   MRCLASS = {11J70 (11K50 47B38 60F05)},
  MRNUMBER = {4169903},
MRREVIEWER = {Faustin\ Adiceam},
       DOI = {10.1007/s10998-020-00325-2},
       URL = {https://doi.org/10.1007/s10998-020-00325-2},
}

@article {LS22,
    AUTHOR = {Sebe, Gabriela Ileana and Lascu, Dan},
     TITLE = {Two asymptotic distributions related to {R}\'enyi-type
              continued fraction expansions},
   JOURNAL = {Period. Math. Hungar.},
  FJOURNAL = {Periodica Mathematica Hungarica. Journal of the J\'anos Bolyai
              Mathematical Society},
    VOLUME = {85},
      YEAR = {2022},
    NUMBER = {2},
     PAGES = {380--398},
      ISSN = {0031-5303,1588-2829},
   MRCLASS = {11J70 (11K50 37A44 60J20)},
  MRNUMBER = {4514193},
MRREVIEWER = {Manuel\ Hauke},
       DOI = {10.1007/s10998-021-00444-4},
       URL = {https://doi.org/10.1007/s10998-021-00444-4},
}

@article {M20,
    AUTHOR = {Mehmetaj, Erblin},
     TITLE = {On the {$r$}-continued fraction expansions of reals},
   JOURNAL = {J. Number Theory},
  FJOURNAL = {Journal of Number Theory},
    VOLUME = {207},
      YEAR = {2020},
     PAGES = {294--314},
      ISSN = {0022-314X,1096-1658},
   MRCLASS = {11J70 (11K50 37A44 37E05)},
  MRNUMBER = {4017948},
MRREVIEWER = {Felipe\ Alberto\ Ram\'irez},
       DOI = {10.1016/j.jnt.2019.07.012},
       URL = {https://doi.org/10.1016/j.jnt.2019.07.012},
}

@article {MO20,
    AUTHOR = {Morier-Genoud, Sophie and Ovsienko, Valentin},
     TITLE = {{$q$}-deformed rationals and {$q$}-continued fractions},
   JOURNAL = {Forum Math. Sigma},
  FJOURNAL = {Forum of Mathematics. Sigma},
    VOLUME = {8},
      YEAR = {2020},
     PAGES = {Paper No. e13, 55},
      ISSN = {2050-5094},
   MRCLASS = {11A55 (05A30 11B57 13F60 57K14)},
  MRNUMBER = {4073883},
MRREVIEWER = {Thomas\ Garrity},
       DOI = {10.1017/fms.2020.9},
       URL = {https://doi.org/10.1017/fms.2020.9},
}

@article {MO19,
    AUTHOR = {Morier-Genoud, Sophie and Ovsienko, Valentin},
     TITLE = {Farey boat: continued fractions and triangulations, modular
              group and polygon dissections},
   JOURNAL = {Jahresber. Dtsch. Math.-Ver.},
  FJOURNAL = {Jahresbericht der Deutschen Mathematiker-Vereinigung},
    VOLUME = {121},
      YEAR = {2019},
    NUMBER = {2},
     PAGES = {91--136},
      ISSN = {0012-0456,1869-7135},
   MRCLASS = {11A55 (05C78 11B75)},
  MRNUMBER = {3955760},
MRREVIEWER = {Mihai\ Cipu},
       DOI = {10.1365/s13291-019-00197-7},
       URL = {https://doi.org/10.1365/s13291-019-00197-7},
}
\end{document}